\documentclass[letterpaper, 10pt, conference]{ieeeconf}
\pdfoutput=1\IEEEoverridecommandlockouts 
\usepackage{amsmath,amssymb,url}
\usepackage{graphicx,subfigure}
\usepackage[pdfpagemode=none,pdfstartview=FitH]{hyperref}

\newcommand{\bracket}[1]{\ensuremath{\left[ #1 \right]}}
\newcommand{\braces}[1]{\ensuremath{\left\{ #1 \right\}}}
\newcommand{\parenth}[1]{\ensuremath{\left( #1 \right)}}

\newcommand{\refeqn}[1]{(\ref{eqn:#1})}
\newcommand{\reffig}[1]{Fig. \ref{fig:#1}}
\newcommand{\tr}[1]{\mathrm{tr}\ensuremath{\negthickspace\bracket{#1}}}
\newcommand{\trs}[1]{\mathrm{tr}\ensuremath{[#1]}}

\newcommand{\SO}{\ensuremath{\mathsf{SO(3)}}}
\newcommand{\T}{\ensuremath{\mathsf{T}}}
\renewcommand{\L}{\ensuremath{\mathsf{L}}}
\newcommand{\so}{\ensuremath{\mathfrak{so}(3)}}

\renewcommand{\Re}{\ensuremath{\mathbb{R}}}

\newcommand{\D}{\ensuremath{\mathbf{D}}}
\newcommand{\Sph}{\ensuremath{\mathsf{S}}}

\newcommand{\Q}{\ensuremath{\mathsf{Q}}}

\title{\LARGE \bf
Robust Adaptive Geometric Tracking Controls on \SO\\with an Application to the Attitude Dynamics of a Quadrotor UAV}

\author{Taeyoung Lee\authorrefmark{1}%
\thanks{Taeyoung Lee, Mechanical and Aerospace Engineering, Florida Institute of Technology, Melbourne, FL 39201 {\tt taeyoung@fit.edu}}%
\thanks{\textsuperscript{\footnotesize\ensuremath{*}}This research has been supported in part by NSF under grants CMMI-1029551.}
}

\newtheorem{definition}{Definition}

\newtheorem{prop}[definition]{Proposition}
\newtheorem{assump}[definition]{Assumption}
\newtheorem{remark}[definition]{Remark}

\begin{document}
\allowdisplaybreaks
\maketitle \thispagestyle{empty} \pagestyle{empty}

\begin{abstract}
This paper provides new results for a robust adaptive tracking control of the attitude dynamics of a rigid body. Both of the attitude dynamics and the proposed control system are globally expressed on the special orthogonal group, to avoid complexities and ambiguities associated with other attitude representations such as Euler angles or quaternions. By designing an adaptive law for the inertia matrix of a rigid body, the proposed control system can asymptotically follow an attitude command without the knowledge of the inertia matrix, and it is extended to guarantee boundedness of tracking errors in the presence of unstructured disturbances. These are illustrated by numerical examples and experiments for the attitude dynamics of a quadrotor UAV.
\end{abstract}

\section{Introduction}

The attitude dynamics of a rigid body appears in various engineering applications, such as aerial and underwater vehicles, robotics, and spacecraft, and the attitude control problem has been extensively studied under various assumptions (see, for example, \cite{WieWeiJGCD89,WenKreITAC91,Sid97,Hug86}).

One of the distinct features of the attitude dynamics is that its configuration manifold is not linear: it evolves on a nonlinear manifold, referred as the special orthogonal group, $\SO$. This yields important and unique properties that cannot be observed from dynamic systems evolving on a linear space. For example, it has been shown that there exists no continuous feedback control system that asymptotically stabilizes an attitude globally on $\SO$~\cite{KodPICDC98,BhaBerSCL00}. 

However, most of the prior work on the attitude control is based on minimal representations of an attitude, or quaternions. It is well known that any minimal attitude representations are defined only locally, and they exhibit kinematic singularities for large angle rotational maneuvers. Quaternions do not have singularities, but they have ambiguities in representing an attitude, as the three-sphere $\Sph^3$ double-covers $\SO$. As a result, in a quaternion-based attitude control system, convergence to a single attitude implies convergence to either of the two disconnected, antipodal points on $\Sph^3$~\cite{MaySanITAC11}. Therefore, depending on the particular choice of control inputs, a quaternion-based control system may become discontinuous when applied to an actual attitude dynamics~\cite{MaySanPACC11b}, and it may also exhibit unwinding behavior, where the controller unnecessarily rotates a rigid body through large angles~\cite{BhaBerSCL00,MaySanPACC11}.

Geometric control is concerned with the development of control systems for dynamic systems evolving on nonlinear manifolds that cannot be globally identified with Euclidean spaces~\cite{Jur97,Blo03,BulLew05}. By characterizing geometric properties of nonlinear manifolds intrinsically, geometric control techniques completely avoids singularities and ambiguities that are associated with local coordinates or improper characterizations of a configuration manifold. This approach has been applied to fully actuated rigid body dynamics on Lie groups to achieve almost global asymptotic stability~\cite{BulLew05,MaiBerITAC06,CabCunPICDC08,ChaMcCITAC09,SanFosJGCD09,LeeLeoPICDC10}.

In this paper, we develop a geometric adaptive controller on $\SO$ to track an attitude and angular velocity command without the knowledge of the inertia matrix of a rigid body. An estimate of the inertia matrix is updated online to provide an asymptotic tracking property. It is also extended to a robust adaptive attitude tracking control system. Stable adaptive control schemes designed without consideration of uncertainties may become unstable in the presence of small disturbances~\cite{IoaSun96}. The presented robust adaptive scheme guarantees the boundedness of the attitude tracking error and the inertia matrix estimation error even if there exist modeling errors or disturbances. Compared with a prior work in~\cite{SanFosJGCD09}, the proposed adaptive tracking control system has simpler structures, and the proposed robust adaptive tracking control system can be applied to unstructured or non-harmonic uncertainties without need for their frequencies.

This paper is organized as follows. We present a global attitude dynamics model in Section \ref{sec:AD}. Adaptive attitude tracking control systems on $\SO$ are developed in Section III, followed by numerical and experimental results.

\section{Attitude Dynamics of a Rigid Body}\label{sec:AD}

We consider the rotational attitude dynamics of a fully-actuated rigid body. We define an inertial reference frame and a body fixed frame whose origin is located at the mass center of the rigid body. The configuration of the rigid body is the orientation of the body fixed frame with respect to the inertial frame, and it is represented by a rotation matrix $R\in\SO$, where the special orthogonal group $\SO$ is the group of $3\times 3$ orthogonal matrices with determinant of one, i.e., $\SO=\{R\in\Re^{3\times 3}\,|\,R^T R=I,\,\det{R}=1\}$.

The equations of motion are given by
\begin{gather}
J\dot \Omega + \Omega\times J\Omega = u + \Delta,\label{eqn:Wdot}\\
\dot R = R\hat\Omega,\label{eqn:Rdot}
\end{gather}
where $J\in\Re^{3\times 3}$ is the inertia matrix in the body fixed frame, and $\Omega\in\Re^3$ and $u\in\Re^3$ are the angular velocity of the rigid body and the control moment, represented with respect to the body fixed frame, respectively. The vector $\Delta\in\Re^3$ represents disturbances caused by either modeling errors or system noises.

The \textit{hat} map $\wedge :\Re^{3}\rightarrow\so$ transforms a vector in $\Re^3$ to a $3\times 3$ skew-symmetric matrix such that $\hat x y = x\times y$ for any $x,y\in\Re^3$. The inverse of the hat map is denoted by the \textit{vee} map $\vee:\so\rightarrow\Re^3$.
Several properties of the hat map are summarized as follows:
\begin{gather}
    \hat x y = x\times y = - y\times x = - \hat y x,\\
%    \hat x^T \hat x = (x^T x) I - x x^T,\\
%    \hat x \hat y \hat x=-(y^Tx)\hat x,\\
%    -\frac{1}{2}\tr{\hat x \hat y} = x^T y,\\
%    \widehat{x\times y} = \hat x \hat y -\hat y \hat x = yx^T-xy^T,\\
    %\tr{\hat x A}=
    \tr{A\hat x }=\frac{1}{2}\tr{\hat x (A-A^T)}=-x^T (A-A^T)^\vee,\label{eqn:hat1}\\
%    \widehat{Ax} = \hat x \parenth{\frac{1}{2}\tr{A}I-A} + \parenth{\frac{1}{2}\tr{A}I-A}^T \hat x,\\
    \hat x  A+A^T\hat x=(\braces{\tr{A}I_{3\times 3}-A}x)^{\wedge},\label{eqn:xAAx}\\
R\hat x R^T = (Rx)^\wedge,\label{eqn:RxR}
\end{gather}
for any $x,y\in\Re^3$, $A\in\Re^{3\times 3}$, and $R\in\SO$. Throughout this paper, the 2-norm of a matrix $A$ is denoted by $\|A\|$, and its Frobenius norm is denoted by $\|A\|_F = \sqrt{\trs{A^TA}}$. We have $\|A\|\leq \|A\|_F\leq \sqrt{r}\|A\|$, where $r$ is the rank of $A$.

\section{Geometric Tracking Control on $\SO$}

We develop adaptive control systems to follow a given smooth attitude command $R_d(t)\in\SO$. The kinematics equation for the attitude command can be written as
\begin{align}
\dot R_d = R_d \hat\Omega_d,\label{eqn:Rddot}
\end{align}
where $\Omega_d\in\Re^3$ is the desired angular velocity. 

\subsection{Attitude Error Dynamics}

One of the important procedures in constructing a control system on a nonlinear manifold $\Q$ is choosing a proper configuration error function, which is a smooth positive definite function $\Psi:\Q\times\Q\rightarrow\Re$ that measures the error between a current configuration and a desired configuration. Once a configuration error function is chosen, a configuration error vector, and a velocity error vector can be defined in the tangent space $\T_q\Q$ by using the derivatives of $\Psi$~\cite{BulLew05}. Then, the remaining procedure is similar to nonlinear control system design in Euclidean spaces: control inputs are carefully designed as a function of these error vectors through a Lyapunov analysis on $\Q$.%, where a Lyapunov candidate also is written in terms of $\Psi$. 

The following form of a configuration error function has been used in~\cite{BulLew05,ChaMcCITAC09}. Here, we summarize its properties developed in those literatures, and we show few additional facts required in this paper.

\begin{prop}\label{prop:1}
For a given tracking command $(R_d,\Omega_d)$, and the current attitude and angular velocity $(R,\Omega)$, we define an attitude error function $\Psi:\SO\times\SO\rightarrow\Re$, an attitude error vector $e_R:\SO\times\SO\rightarrow\Re^3$, and an angular velocity error vector $e_\Omega:\SO\times\Re^3\times\SO\times\Re^3\rightarrow \Re^3$ as follows:
\begin{gather}
\Psi (R,R_d) = \frac{1}{2}\tr{G(I-R_d^TR)},\label{eqn:Psi}\\
e_R(R,R_d) =\frac{1}{2} (GR_d^TR-R^TR_dG)^\vee,\label{eqn:eR}\\
e_\Omega(R,\Omega,R_d,\Omega) = \Omega - R^T R_d\Omega_d,\label{eqn:eW}
\end{gather}
where the matrix $G\in\Re^{3\times 3}$ is given by $G=\mathrm{diag}[g_1,g_2,g_3]$ for distinct, positive constants $g_1,g_2,g_3\in\Re$. Then, the following statements hold:
\begin{itemize}
\item[(i)] $\Psi$ is locally positive definite about $R=R_d$.
\item[(ii)] the left-trivialized derivative of $\Psi$ is given by
\begin{align}
\T^*_I \L_R\, (\D_R\Psi(R,R_d))= e_R.
\end{align}
\item[(iii)] the critical points of $\Psi$, where $e_R=0$, are $\{R_d\}\cup\{R_d\exp (\pi \hat s)\}$ for $s\in\{e_1,e_2,e_3\}$.
\item[(iv)] a lower bound of $\Psi$ is given as follows:
\begin{align}
b_1\|e_R(R,R_d)\|^2 \leq \Psi(R,R_d),\label{eqn:eRPsi}
\end{align}
where the constant $b_1$ is given by $b_1=\frac{h_1}{h_2+h_3}$ for 
\begin{align*}
h_1&= \min\{g_1+g_2,g_2+g_3,g_3+g_1 \},\\
h_2&=\max\{(g_1-g_2)^2,(g_2-g_3)^2,(g_3-g_1)^2\},\\
h_3&=\max\{(g_1+g_2)^2,(g_2+g_3)^2,(g_3+g_1)^2\}.
\end{align*}
\item[(v)] Let $\psi$ be a positive constant that is strictly less than $h_1$. If $\Psi(R,R_d)< \psi<h_1$, then an upper bound of $\Psi$ is given by
\begin{align}
\Psi(R,R_d)\leq b_2 \|e_R(R,R_d)\|^2,\label{eqn:eRPsi2}
\end{align}
where the constant $b_2$ is given by $b_2=\frac{h_1h_4}{h_5(h_1-\psi)}$ for 
\begin{align*}
h_4 & = \max\{g_1+g_2,g_2+g_3,g_3+g_1\}\\
h_5 & = \min\{(g_1+g_2)^2,(g_2+g_3)^2,(g_3+g_1)^2\}.
\end{align*}
\end{itemize}
\end{prop}
\begin{proof}
The proofs of (i)-(iii) are available at~\cite[(Chap. 11)]{BulLew05}. To show (iv) and (v), let $Q=R_d^TR=\exp\hat x\in\SO$ for $x\in\Re^3$ from Rodrigues' formula. Using the Matlab Symbolic Computation Tool, we find 
\begin{align*}
\Psi & = \frac{1-\cos\|x\|}{2\|x\|^2}\sum_{(i,j,k)\in\mathcal{C}} (g_i+g_j)x_k^2,\\
\|e_R\|^2 & = \frac{(1-\cos\|x\|)^2}{4\|x\|^4}\sum_{(i,j,k)\in\mathcal{C}} (g_i-g_j)^2 x_i^2 x_j^2\\
&\quad + \frac{\sin^2\|x\|}{4\|x\|^2}\sum_{(i,j,k)\in\mathcal{C}} (g_i+g_j)^2 x_k^2,
\end{align*}
where $\mathcal{C}=\{(1,2,3),(2,3,1),(3,1,2)\}$. When $\Psi=0$, \refeqn{eRPsi} is trivial. Assuming that $\Psi\neq 0$, therefore $\|x\|\neq 0$, an upper bound of $\frac{\|e_R\|^2}{\Psi}$ is given by
\begin{align*}
\frac{\|e_R\|^2}{\Psi}& \leq \frac{1}{2h_1}(1-\cos\|x\|)h_2+\frac{1}{2h_1}(1+\cos\|x\|)h_3\\
&\leq \frac{h_2+h_3}{h_1},
\end{align*}
which shows \refeqn{eRPsi}.

Next, we consider (v). When $\Psi=0$, \refeqn{eRPsi2} is trivial. Hereafter, we assume $\Psi\neq 0$, therefore $R\neq R_d$. At the three remaining critical points of $\Psi$, the values of $\Psi$ are given by $g_1+g_2$, $g_2+g_3$, or $g_3+g_1$. So, from the given bound $\Psi<\psi$, these three critical points are avoided, and we can guarantee that $e_R\neq 0$ and $\|x\|<\pi$. An upper bound of $\frac{\Psi}{\|e_R\|^2}$ is given by
\begin{align}
\frac{\Psi}{\|e_R\|^2}&\leq \frac{2(1-\cos\|x\|)}{\sin\|x\|^2}\frac{\sum_{\mathcal{C}} (g_i+g_j)x_k^2/\|x\|^2}{\sum_{\mathcal{C}} (g_i+g_j)^2x_k^2/\|x\|^2}\nonumber\\
&\leq \frac{2}{1+\cos\|x\|}\frac{h_4}{h_5},\label{eqn:ratio0}
\end{align}
Also, an upper bound of $h_1-\psi$ is given by
\begin{align*}
h_1-\psi < h_1-\Psi \leq h_1 - \frac{1-\cos\|x\|}{2}h_1=\frac{h_1}{2}(1+\cos\|x\|).
\end{align*}
Substituting this into \refeqn{ratio0}, we have
\begin{align*}
\frac{\Psi}{\|e_R\|^2}&\leq \frac{h_4h_1}{h_5(h_1-\psi)}=b_2,
\end{align*}
which shows \refeqn{eRPsi2}.
\end{proof}

The corresponding attitude error dynamics for the attitude error function $\Psi$, the attitude error vector $e_R$, and the angular velocity error $e_\Omega$ are as follows.
\begin{prop}\label{prop:2}
The error dynamics for $\Psi$, $e_R$, $e_\Omega$ satisfies
\begin{gather}
\frac{d}{dt} (R_d^T R) =  R_d^T R \hat e_\Omega\label{eqn:RdTRdot}\\
\frac{d}{dt}(\Psi(R,R_d))  = e_R\cdot e_\Omega,\label{eqn:Psidot}\\
\dot e_R  = E(R,R_d)e_\Omega,\label{eqn:eRdotE}\\
\dot e_\Omega  = J^{-1}(-\Omega\times J\Omega + u+\Delta)- \alpha_d,\label{eqn:eWdot}
\end{gather}
where the matrix $E(R,R_d)\in\Re^{3\times 3}$, and the angular acceleration $\alpha_d\in\Re^3$, that is caused by the attitude command, and measured in the body fixed frame, are given by
\begin{align}
E(R,R_d)&=\frac{1}{2}(\trs{R^T R_dG}I -R^T R_dG),\label{eqn:E}\\
\alpha_d&=-\hat\Omega R^T R_d\Omega_d+ R^T R_d{\dot \Omega}_d.\label{eqn:ad}
\end{align}
Furthermore, the matrix $E(R,R_d)$ is bounded by
\begin{align}
\|E(R,R_d)\| \leq \frac{1}{\sqrt{2}}\tr{G}.\label{eqn:Eb}
\end{align}
\end{prop}

\begin{proof}
From the kinematics equations \refeqn{Rdot}, \refeqn{Rddot}, the time-derivative of $R_d^T R$ is given by
\begin{align*}
\frac{d}{dt}(R_d^T R) = -\hat\Omega_dR_d^T R + R_d^T R\hat\Omega.
\end{align*}
Using \refeqn{RxR}, this can be written as
\begin{align*}
\frac{d}{dt}(R_d^T R) = R_d^T R(-(R^T R_d\Omega_d)^\wedge + \hat\Omega),
\end{align*}
which shows \refeqn{RdTRdot}. From this, the time-derivative of the attitude error function is given by
\begin{align*}
\frac{d}{dt} \Psi (R,R_d) &  = -\frac{1}{2}\tr{GR_d^T R\hat e_\Omega}
\end{align*}
Applying \refeqn{hat1}, \refeqn{eR} into this, we obtain \refeqn{Psidot}. Next, the time-derivative of the attitude error vector is given by
\begin{align*}
\dot e_R = \frac{1}{2}(G R_d^T R \hat e_\Omega + \hat e_\Omega R^T R_d G)^\vee.
\end{align*}
Using the properties of the hat map, given by \refeqn{xAAx}, this can be further reduced to \refeqn{eRdotE} and \refeqn{E}. 

To show \refeqn{Eb}, we find the Frobenius norm $\|E\|_F$:
\begin{align}
\|E(R,R_d)\|_F = \sqrt{\tr{E^TE}}= \frac{1}{2}\sqrt{\trs{G^2} + \tr{R^TR_d G}^2},\label{eqn:EF0}
\end{align}
where we use the facts that $\trs{AB}=\trs{BA}$ and $\trs{cA}=c\trs{A}$ for any matrices $A,B$, and a constant $c$. Let $Q=R_d^TR=\exp\hat x\in\SO$ for $x\in\Re^3$ from Rodrigues' formula. Using the Matlab Symbolic Computation Tool, we find 
\begin{align*}
\tr{R^TR_d G} = \cos\|x\| \sum_{i=1}^3 g_i(1-\frac{x_i^2}{\|x\|^2}) + \sum_{i=1}^3 g_i\frac{x_i^2}{\|x\|^2},
\end{align*}
 Since $0\leq x_i^2/\|x\|^2\leq 1$, we have $\trs{R^TR_d G}\leq \sum_{i=1}^3 g_i=\trs{G}$. Substituting this into \refeqn{EF0}, we obtain
\begin{align*}
\|E(R,R_d)\|_F^2 \leq \frac{1}{4}(\tr{G^2}+\tr{G}^2) \leq \frac{1}{2}\tr{G}^2,
\end{align*}
which shows \refeqn{Eb}, since $\|E\|\leq \|E\|_F$. 

From \refeqn{Wdot}, \refeqn{Rdot}, \refeqn{Rddot}, and using the fact that $\hat\Omega_d\Omega_d=\Omega_d\times\Omega_d=0$ for any $\Omega_d\in\Re^3$, the time-derivative of the angular velocity error $e_\Omega$ is given by
\begin{align*}
\dot e_\Omega & = \dot \Omega +\hat\Omega R^T R_d\Omega_d-R^T R_d\hat\Omega_d\Omega _d - R^T R_d{\dot \Omega}_d\nonumber\\
& =J^{-1}(-\Omega\times J\Omega + u+\Delta)+\hat\Omega R^T R_d\Omega_d- R^T R_d{\dot \Omega}_d,\nonumber\\
& =J^{-1}(-\Omega\times J\Omega + u+\Delta)-\alpha_d,
\end{align*}
where $\alpha_d$ is given by \refeqn{ad}.
\end{proof}

\subsection{Adaptive Attitude Tracking}

Attitude tracking control systems require the knowledge of an inertia matrix when the given attitude command is not fixed. But, it is difficult to measure the value of an inertia matrix exactly. In general, there is an estimation error, given by
\begin{align}
\tilde J = J - \bar J,\label{eqn:tildeJ}
\end{align}
where the exact inertia matrix and its estimate are denoted by the matrices $J$ and $\bar J\in\Re^{3\times 3}$, respectively. All of matrices, $J$, $\bar J$, $\tilde J$ are symmetric.

Here, an adaptive tracking controller for the attitude dynamics of a rigid body is presented to follow a given attitude command without the knowledge of its inertia matrix assuming that there is no disturbance, and that the bounds of the inertia matrix are given.

\begin{assump}\label{assump:J}
The minimum eigenvalue $\lambda_m\in\Re$, and the maximum eigenvalue $\lambda_M\in\Re$ of the true inertia matrix $J$ given at \refeqn{Wdot} are known.
\end{assump}

\begin{prop}\label{prop:AT}
Assume that there is no disturbance in the attitude dynamics, i.e. $\Delta = 0$ at \refeqn{Wdot}, and Assumption \ref{assump:J} is satisfied. For a given attitude command $R_d(t)$, and positive constants $k_R,k_\Omega,k_J\in\Re$, we define a control input $u\in\Re^3$, and an update law for $\bar J$ as follows:
\begin{align}
u & = -k_R e_R - k_\Omega e_\Omega + \Omega\times \bar{J}\Omega +\bar{J} \alpha_d,\label{eqn:u}\\
\dot{\bar J} & = \frac{k_J}{2} (-\alpha_d e_A^T - e_A\alpha_d^T +\Omega\Omega^T \hat e_A -\hat e_A \Omega\Omega^T),\label{eqn:Jdot}
\end{align}
where $e_A\in\Re^3$ is an augmented error vector given by
\begin{align}
e_A=e_\Omega + c e_R\label{eqn:eA}
\end{align}
for a positive constant $c$ satisfying
\begin{align}
c < \min\braces{
\sqrt{\frac{2b_1k_R\lambda_{m}}{\lambda_M^2}},\,
\frac{\sqrt{2}k_\Omega}{\lambda_M\trs{G}},\,
\frac{4 k_R k_\Omega }{k_\Omega^2 + \frac{1}{\sqrt 2} k_R \lambda_M\trs{G}}}.\label{eqn:c0}
\end{align}
Then, the zero equilibrium of the tracking errors $(e_R,e_\Omega)$ and the estimation error $\tilde J$ is stable, and those errors are uniformly bounded. Furthermore, the tracking errors for the attitude and the angular velocity asymptotically converge to zero, i.e. $e_R,e_\Omega\rightarrow 0$ as $t\rightarrow\infty$.
\end{prop}

\begin{proof}
Consider the following Lyapunov function:
\begin{align}
\mathcal{V} = \frac{1}{2} e_\Omega \cdot J e_\Omega + k_R \Psi(R,R_d) + c Je_\Omega\cdot e_R + \frac{1}{2k_J}\|\tilde J\|_F^2.\label{eqn:V}
\end{align}
From \refeqn{eRPsi}, we obtain
\begin{align}
z^T W_{11} z \leq \mathcal{V} %\leq z^T W_{12} z+ \frac{1}{2k_J}\tr{\tilde J ^T\tilde J},\label{eqn:Vb}
\end{align}
where $z = [\|e_R\|;\,\|e_\Omega\|;\|\tilde J\|_F]\in\Re^3$, and the matrix $W_{1}\in\Re^{3\times 3}$ are given by
\begin{align}
W_{11} = \begin{bmatrix}  b_1k_R & \frac{1}{2}c\lambda_{M}& 0\\
\frac{1}{2}c\lambda_{M} & \frac{1}{2}\lambda_{m}& 0\\
0&0&\frac{1}{2k_J}
\end{bmatrix}.\label{eqn:W11}%,\;
%W_{12} = \begin{bmatrix}  2k_R & \frac{1}{2}c_2\lambda_{M}\\
%\frac{1}{2}c_2\lambda_{M} & \frac{1}{2}\lambda_{M}
%\end{bmatrix}.
\end{align}

Substituting \refeqn{u} into \refeqn{eWdot} with $\Delta=0$, we obtain
\begin{align}
J\dot e_\Omega & = -\Omega\times J\Omega -J\alpha_d -k_R e_R - k_\Omega e_\Omega + \Omega\times \bar J\Omega +\bar J \alpha_d,\nonumber\\
& =-k_R e_R - k_\Omega e_\Omega -\tilde J \alpha_d -\Omega\times \tilde J \Omega.\label{eqn:JeWdot}
\end{align}
Using \refeqn{Psidot}, \refeqn{eRdotE}, \refeqn{JeWdot}, the time-derivative of $\mathcal{V}$ is given by
\begin{align*}
\dot{\mathcal{V}} 
& = e_\Omega\cdot(-k_R e_R - k_\Omega e_\Omega -\tilde J \alpha_d -\Omega\times \tilde J \Omega)\\
&\quad + k_R e_R\cdot e_\Omega + c(-k_R e_R - k_\Omega e_\Omega -\tilde J \alpha_d -\Omega\times \tilde J \Omega)\cdot e_R\\
&\quad + c Je_\Omega\cdot Ee_\Omega
+\frac{1}{k_J}\tr{\tilde J \dot{\tilde J}}\\
& = -k_\Omega \|e_\Omega\|^2 -ck_R\|e_R\|^2 + c Je_\Omega\cdot Ee_\Omega-ck_\Omega e_\Omega\cdot e_R  \\
&\quad -(e_\Omega+ce_R)\cdot( \tilde J \alpha_d +\Omega\times \tilde J \Omega)
+\frac{1}{k_J}\tr{\tilde J \dot{\tilde J}}.
\end{align*}
From \refeqn{eA}, and using the fact that $x\cdot y=\mathrm{tr}[xy^T]=\mathrm{tr}[yx^T]$ for any $x,y\in\Re^3$, and the scalar triple product identity, this can be written as
\begin{align*}
\dot{\mathcal{V}}
& = -k_\Omega \|e_\Omega\|^2 -ck_R\|e_R\|^2
+ c Je_\Omega\cdot Ee_\Omega-ck_\Omega e_\Omega\cdot e_R  \\
&\quad 
+\tr{\tilde J \braces{-\alpha_de_A^T-\Omega(e_A\times\Omega)^T +\frac{1}{k_J} \dot{\tilde J}}}.
\end{align*}
Since $\dot{\tilde J}=-\dot{\bar J}$, we can substitute \refeqn{Jdot} into this. Using the facts that $\mathrm{tr}[\tilde J A]=\mathrm{tr}[\tilde J A^T]$ for any $A\in\Re^{3\times 3}$, and $(e_A\times \Omega)^T = (\hat e_A \Omega)^T=-\Omega^T\hat e_A$, it reduces to
\begin{align}
\dot{\mathcal{V}}
& = -k_\Omega \|e_\Omega\|^2 -ck_R\|e_R\|^2
+ c Je_\Omega\cdot Ee_\Omega-ck_\Omega e_\Omega\cdot e_R.\label{eqn:Vdot00}
\end{align}
From \refeqn{Eb}, it is bounded by
\begin{align}
\dot{\mathcal{V}} & \leq -(k_\Omega-\frac{c}{\sqrt{2}}\lambda_M\trs{G}) \|e_\Omega\|^2
-ck_R\|e_R\|^2\nonumber\\&\quad 
+ck_\Omega \|e_\Omega\|\|e_R\| = -\zeta^T W_2 \zeta,\label{eqn:Vdot}
\end{align}
where $\zeta=[\|e_R\|;\,\|e_\Omega\|]\in\Re^2$, and the matrix $W_2\in\Re^{2\times 2}$ is given by
\begin{align}
W_2 = \begin{bmatrix}  c k_R & -\frac{c k_\Omega}{2}\\
- \frac{c k_\Omega}{2} & k_\Omega - \frac{c}{\sqrt{2}}\lambda_M\trs{G}\end{bmatrix}.\label{eqn:W2}
\end{align}
The inequality \refeqn{c0} for the constant $c$ guarantees that the matrices $W_{11},W_{2}$ are positive definite. 

This implies that the Lyapunov function $\mathcal{V}(t)$ is bounded from below and it is nonincreasing. Therefore, it has a limit, $\lim_{t\rightarrow\infty} \mathcal{V}(t)=\mathcal{V}_\infty$, and $e_R,e_\Omega,\bar J\in\mathcal{L}_\infty$.\footnote{A function $f:\Re\rightarrow\Re$ belongs to the $\mathcal{L}_p$ space for $p\in[1,\infty)$, if the following $p$-norm of the function exits, $\|f\|_p=\braces{\int_{0}^\infty |f(\tau)|^p \,d\tau}^{1/p}$. } From \refeqn{eRdotE}, \refeqn{JeWdot}, we have $\dot e_R,\dot e_\Omega\in\mathcal{L}_\infty$. Furthermore $e_R,e_\Omega\in\mathcal{L}_2$ since $\int_{0}^\infty \zeta(\tau)^T W_2 \zeta(\tau)d\tau \leq \mathcal{V}(0)-\mathcal{V}_\infty<\infty$. According to Barbalat's lemma (or Lemma 3.2.5 in \cite{IoaSun96}), we have $e_R,e_\Omega\rightarrow 0$ as $t\rightarrow\infty$.
\end{proof}

\begin{remark}
This proposition guarantees that the attitude error vector $e_R$ asymptotically converges to zero. But, this does not necessarily imply that $R\rightarrow R_d$ as $t\rightarrow\infty$. According to Proposition \ref{prop:1}, there exist three additional critical points of $\Psi$, namely $\{R_d\exp(\pi\hat s)\}$ for $s\in\{e_1,e_2,e_3\}$, where $e_R=0$. This is due to the nonlinear structures of $\SO$, and these cannot be avoided for any continuous control systems~\cite{KodPICDC88,BhaBerSCL00}.

But, we can show that those three additional equilibrium points are unstable, by using linearization or by showing that the hessian of $\Psi$ is indefinite at those points. It turned out that these points are saddle equilibria, which have both of stable manifolds and unstable manifolds~\cite{LeeLeoPICDC11}. The union of the stable manifolds to these undesirable equilibria has a lower dimension than the tangent bundle of the configuration space, and we say that it has an \textit{almost}-global stabilization property.
\end{remark}

\begin{remark}
At Assumption \ref{assump:J}, the minimum eigenvalue $\lambda_m$ and the maximum eigenvalue $\lambda_M$ of the inertia matrix $J$ are required. But, in Proposition \ref{prop:AT}, they are only used to find the coefficient $c$ at \refeqn{c0}. So, Assumption \ref{assump:J} can be relaxed as requiring an upper bound of $\lambda_m$ and a lower bound of $\lambda_M$, which are relatively simpler to estimate.
\end{remark}

\subsection{Robust Adaptive Attitude Tracking}

The adaptive tracking control system developed in the previous section is based on the assumption that there is no disturbance in the attitude dynamics. But, it has been discovered that adaptive control schemes may become unstable in the presence of small disturbances~\cite{IoaSun96}. Robust adaptive control deals with redesigning or modifying adaptive control schemes to make them robust with respect to unmodeled dynamics or bounded disturbances. In this section, we develop a robust adaptive attitude tracking control system assuming that the bound of disturbances are given.

\begin{assump}\label{assump:Delta}
The disturbance term in the attitude dynamics at \refeqn{Wdot} is bounded by a known constant, i.e. $\|\Delta\|\leq \delta $ for a given positive constant $\delta$.
\end{assump}

\begin{prop}\label{prop:RAC}
Suppose that Assumptions \ref{assump:J} and \ref{assump:Delta} hold. For a given attitude command $R_d(t)$, and positive constants $k_R,k_\Omega,k_J,\sigma,\epsilon\in\Re$, we define a control input $u\in\Re^3$, and an update law for $\bar J$ as follows:
\begin{align}
u & = -k_R e_R - k_\Omega e_\Omega + \Omega\times \bar{J}\Omega +\bar{J} \alpha_d+v,\label{eqn:u2}\\
v & = -\frac{\delta^2 e_A}{\delta\|e_A\|+\epsilon},\label{eqn:v}\\
\dot{\bar J} & = \frac{k_J}{2} (-\alpha_d e_A^T - e_A\alpha_d^T +\Omega\Omega^T \hat e_A -\hat e_A \Omega\Omega^T-2\sigma \bar J),\label{eqn:Jdot2}
\end{align}
where $e_A\in\Re^3$ is an augmented error vector given at \refeqn{eA} for a positive constant $c$ satisfying \refeqn{c0}. Then, if $\sigma$ and $\epsilon$ are sufficiently small, the zero equilibrium of the tracking errors $(e_R,e_\Omega)$ and the estimation error $\tilde J$ are uniformly bounded. 
\end{prop}

\begin{proof}
Consider the Lyapunov function $\mathcal{V}$ at \refeqn{V}. For a positive constant $\psi<h_1$, define $D\subset\SO$ as
\begin{align*}
D=\{R\in\SO\,|\,\Psi < \psi < h_1\}
\end{align*}
From Proposition \ref{prop:1}, the Lyapunov function is bounded in $D$ by
\begin{align}
z^T W_{11} z \leq \mathcal{V} \leq z^T W_{12} z,\label{eqn:Vb}
\end{align}
where $z = [\|e_R\|;\,\|e_\Omega\|;\,\|\tilde J\|_F]\in\Re^2$, the matrix $W_{11}\in\Re^{2\times 2}$ is given by \refeqn{W11}, and the matrix $W_{12}$ is given by
\begin{align*}
W_{12} = \begin{bmatrix}  b_2k_R & \frac{1}{2}c_2\lambda_{M}& 0\\
\frac{1}{2}c_2\lambda_{M} & \frac{1}{2}\lambda_{M}&0\\
0 & 0& \frac{1}{2k_J}
\end{bmatrix}.
\end{align*}

The time-derivative of $\mathcal{V}$ along the presented control inputs is written as
\begin{align}
\dot{\mathcal{V}}
& = -k_\Omega \|e_\Omega\|^2 -ck_R\|e_R\|^2
+ c Je_\Omega\cdot Ee_\Omega -ck_\Omega e_\Omega\cdot e_R \nonumber\\
&\quad+e_A\cdot(\Delta+v) +\sigma\tr{\tilde J\bar J}.\label{eqn:Vdot200}
\end{align}
Compared with \refeqn{Vdot00}, this has three additional terms caused by $\Delta, v$ and $\sigma$. From Assumption \ref{assump:Delta} and \refeqn{v}, the second last term of \refeqn{Vdot200} is bounded by
\begin{align}
e_A\cdot(\Delta+v) & \leq \delta\|e_A\| - \frac{\delta^2\|e_A\|^2}{\delta\|e_A\|+\epsilon} =  \frac{\delta\|e_A\|}{\delta\|e_A\|+\epsilon}\epsilon \leq \epsilon.\label{eqn:B00}
\end{align}
The last term of \refeqn{Vdot200} is bounded by
\begin{align*}
\tr{\tilde J\bar J} & = \tr{\tilde J (J-\tilde J)} =\sum_{1\leq i,j\leq 3} (-\tilde J_{ij}^2 +J_{ij}\tilde J_{ij})\\
&\leq \sum_{1\leq i,j\leq 3} (-\frac{1}{2} \tilde J_{ij}^2 +\frac{1}{2} J_{ij}^2)
=-\frac{1}{2}\tr{\tilde J^2} + \frac{1}{2}\tr{J^2}\\
& = -\frac{1}{2}\|\tilde J\|_F^2 + \frac{1}{2}\|J\|_F^2.
\end{align*}
Using the relation between a Frobenius norm and a matrix 2-norm, we have $\|J\|_F\leq \sqrt{3}\|J\| =\sqrt{3}\lambda_M$. Therefore,
\begin{align}
\tr{\tilde J\bar J} \leq -\frac{1}{2}\|\tilde J\|_F^2 +\frac{3}{2}\lambda_M^2.\label{eqn:B01}
\end{align}
Substituting \refeqn{B00}, \refeqn{B01} into \refeqn{Vdot200}, we obtain
\begin{align}
\dot{\mathcal{V}} & \leq -z^T W_3 z +\frac{3}{2}\sigma\lambda_M^2 + \epsilon \label{eqn:Vdot20}
\end{align}
where the matrix $W_3\in\Re^{3\times 3}$ is given by
\begin{align}
W_3 = \begin{bmatrix}  c k_R & -\frac{c k_\Omega}{2}&0\\
- \frac{c k_\Omega}{2} & k_\Omega - \frac{c}{\sqrt{2}}\lambda_M\trs{G}&0\\
0&0&\frac{1}{2}\sigma\end{bmatrix}.\label{eqn:W3}
\end{align}
The inequality \refeqn{c0} for the constant $c$ guarantees that the matrices $W_{11},W_{12},W_{3}$ become positive definite. Then, we have
\begin{align}
\dot{\mathcal{V}}\leq -\frac{\lambda_{\min}(W_2)}{\lambda_{\max}(W_{12})} \mathcal{V}+\frac{3}{2}\sigma\lambda_M^2 + \epsilon,\label{eqn:Vdot2}
\end{align}
where $\lambda_{\min}{(\cdot)}$ and $\lambda_{\max}(\cdot)$ represent the minimum eigenvalue and the maximum eigenvalue of a matrix, respectively. This implies that $\dot{\mathcal{V}}<0$ when $\mathcal{V} > \frac{\lambda_{\max}(W_{12})}{\lambda_{\min}(W_3)}(\frac{3}{2}\sigma\lambda_M^2 + \epsilon)\triangleq d_1$.

\setlength{\unitlength}{0.1\columnwidth}
\begin{figure}
\scriptsize\selectfont
\centerline{
\begin{picture}(4.5,4.9)(0,0)
\put(0,0){\includegraphics[width=0.45\columnwidth]{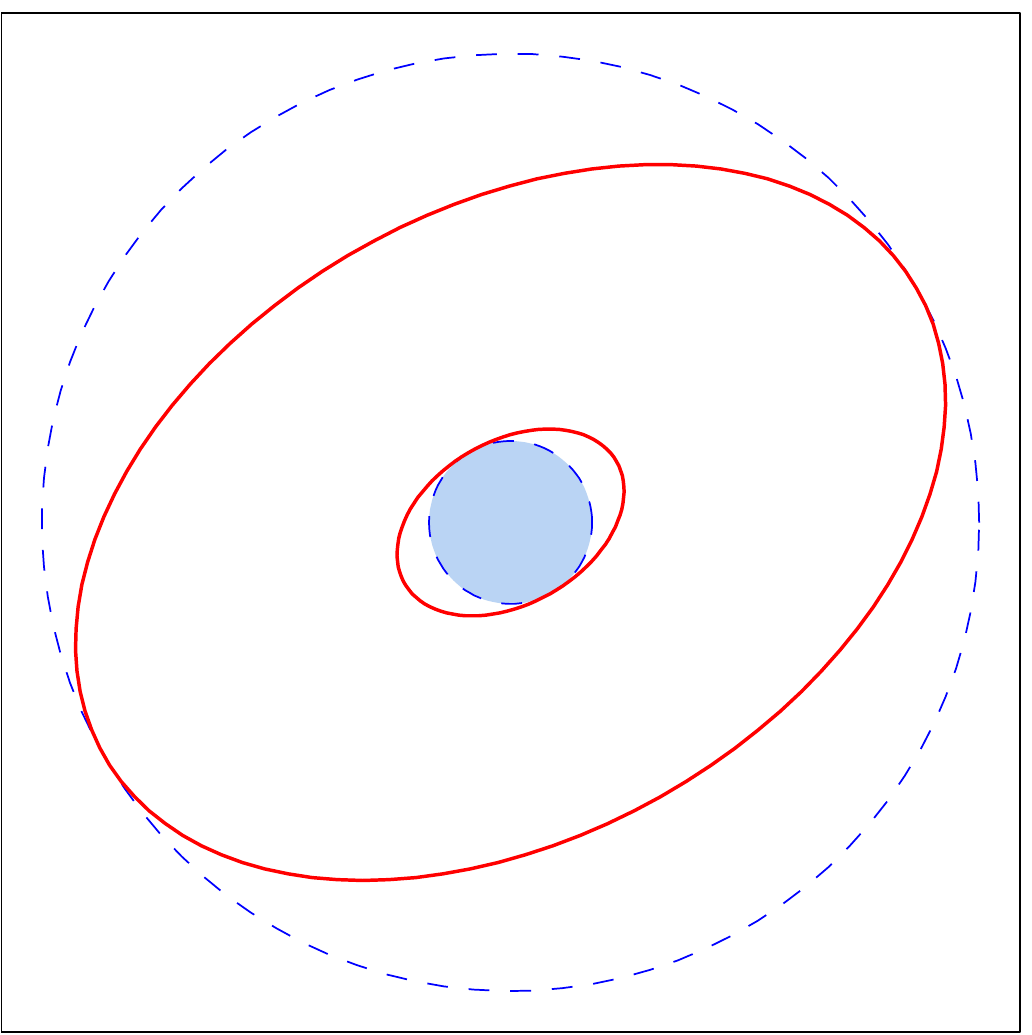}}
\put(0.1,4.0){$B_{\psi/b_2}$}
\put(0,4.6){$D\times\Re^3\times\Re^{3\times 3}$}
\put(3.1,3.45){$L_{d_2}$}
\put(2.64,2.6){$L_{d_1}$}
\end{picture}}
\caption{Boundedness of the error: Outside of the shaded region, represented by $\{\lambda_{\min}(W_3)\|z\|^2\geq (\frac{3}{2}\sigma\lambda_M^2 + \epsilon)\}$, we have $\dot{\mathcal{V}}\leq 0$ from \refeqn{Vdot20}. Inside of the larger ball, $B_{\psi/b_2}=\{ \|z\|^2 \leq \psi/b_2 \}\subset D\times \Re^3\times\Re^{3\times 3}$, equations \refeqn{Vb} and \refeqn{Vdot2} hold. The inequality \refeqn{d1d2} guarantees that the smallest sublevel set $L_{d_1}$ of $\mathcal{V}$, covering the shaded area, lies inside of the largest sublevel set $L_{d_2}$ of $\mathcal{V}$ in $B_{\psi/b_2}$, i.e. $L_{d_1}\subset L_{d_2}$. Therefore, along any solution starting in $L_{d_2}$, $\mathcal{V}$ decreases until the solution enters $L_{d_1}$, thereby yielding uniform boundedness.}
\end{figure}
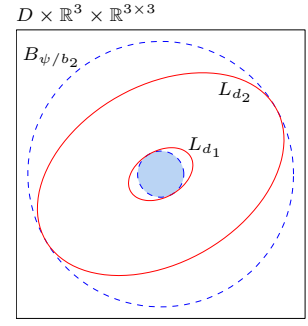

Let a sublevel set of $\mathcal{V}$ be $L_\gamma=\{(R,\Omega,\bar J)\in \SO\times\Re^3\times\Re^{3\times 3}\}\,|\, \mathcal{V}\leq \gamma\}$ for a constant $\gamma>0$. If the following inequality for $\gamma$ is satisfied
\begin{align*}
\gamma < \frac{\psi}{b_2}\lambda_{\min}(W_{11})\triangleq d_2,
\end{align*}
we can guarantee that $L_\gamma\subset D\times\Re^3\times\Re^{3\times 3}$, since it implies that $\|z\|^2 < \frac{\psi}{b_2}$, which leads $\Psi\leq b_2\|e_R\|^2\leq b_2\|z\|^2<\psi$. 

Then, from \refeqn{Vdot2}, a sublevel set $L_\gamma$ is a positively invariant set, when $d_1<\gamma<d_2$, and it becomes smaller until $\gamma=d_1$. In order to guarantee the existence of such $L_{\gamma}$, the following inequality should be satisfied
\begin{align}
d_1=\frac{\lambda_{\max}(W_{12})}{\lambda_{\min}(W_3)}(\frac{3}{2}\sigma\lambda_M^2 + \epsilon) 
< \frac{\psi}{b_2}\lambda_{\min}(W_{11})=d_2,\label{eqn:d1d2}
\end{align}
which can be achieved by choosing sufficiently small $\sigma$ and $\epsilon$. Then, according to Theorem 5.1 in~\cite{Kha96}, for any initial condition satisfying $\mathcal{V}(0)< d_2$, its solution exponentially converges to the following set:
\begin{align*}
L_{d_1}\subset \braces{\|z\|^2 \leq \frac{\lambda_{\max}(W_{12})}{\lambda_{\min}(W_{11})\lambda_{\min}(W_2)}\parenth{\frac{3}{2}\sigma\lambda_M^2 + \epsilon}}.
\end{align*}
\end{proof}

\begin{remark}
The robust adaptive control system in Proposition \ref{prop:RAC} is referred to as fixed $\sigma$-modification \cite{IoaSun96}, where robustness is achieved at the expense of replacing the asymptotic tracking property of Proposition \ref{prop:AT} by boundedness. This property can be improved by the following approaches: (i) the leakage term $-2\sigma \bar J$ at \refeqn{Jdot2} can be replaced by $-2\sigma (\bar J-J^\star)$, where $J^*$ denotes the best possible prior estimate of the inertia matrix. This shifts the tendency of $\bar J$ from zero to $J^\star$, thereby reducing the ultimate bound, (ii) a switching $\sigma$-modification or $\epsilon_1$-modification can be used to improve the convergence properties in the expense of discontinuities, (iii) the constant $\epsilon$ at \refeqn{v} can be replaced by $\epsilon\exp(-\beta t)$ for any $\beta >0$ to reduce the ultimate bound. The corresponding stability analyses are similar to the presented case, and they are deferred to a future study.
\end{remark}

%\subsection{Properties}
%
%One of the unique properties of the presented adaptive control systems is that they are directly developed on $\SO$ using rotation matrices. Therefore, they avoid complexities and singularities associated with local coordinates of $\SO$, such as Euler angles. They also avoid the ambiguities that arise when using quaternions to represent the attitude dynamics. As the three-sphere $\Sph^3$ double covers $\SO$, any attitude feedback controller designed in terms of quaternions could yield different control inputs depending on the choice of quaternion vectors. The corresponding stability analysis would need to carefully consider the fact that convergence to a single attitude implies convergence to either of the two disconnected, antipodal points on $\Sph^3$~\cite{MaySanPICDC09}. This requires a continuous selection of the sign of quaternions or a discontinuous control system, which are shown to be sensitive to small measurement noise~\cite{SanMesPACC06}. Without these considerations, a quaternion-based controller can exhibit an unwinding phenomenon, where the controller unnecessarily rotates the attitude through large angles~\cite{BhaBerSCL00}. In this paper, the use of rotation matrices in the controller design and stability analysis completely eliminates these difficulties.

\section{Numerical Examples}

Parameters of a rigid body model and control systems are chosen as follows\footnote{All of variables are defined in kilograms, meters, seconds, and radians}:
\begin{gather*}
J=\begin{bmatrix}
   1.059\times 10^{-2}&  -5.156\times 10^{-6}&   2.361\times 10^{-5}\\
  -5.156\times 10^{-6}&   1.059\times 10^{-2}&  -1.026\times 10^{-5}\\
   2.361\times 10^{-5}&  -1.026\times 10^{-5}&   1.005\times 10^{-2}\\
\end{bmatrix},\\
k_R=0.0424,\quad k_\Omega=0.0296,\quad k_J=0.1,\\
c=1.0,\quad\sigma=0.01,\quad\epsilon=0.002,\quad\delta=0.2.
\end{gather*}
Initial conditions are given by
\begin{gather*}
\bar J(0)=0.001I, \quad R(0)=I,\quad \Omega(0)=0.
\end{gather*}
The desired attitude command is described by using 3-2-1 Euler angles~\cite{ShuJAS93}, i.e. $R_d(t)=R_d(\phi(t),\theta(t),\psi(t))$, and these angles are chosen as
\begin{align*}
\phi(t) =  \frac{\pi}{9}\sin(\pi t), \quad \theta(t)=\frac{\pi}{9}\cos(\pi t),\quad \psi(t)=0.
\end{align*}

We consider three cases:
\begin{itemize}
\item[(i)] Adaptive attitude tracking control system presented at Proposition \ref{prop:AT} without disturbances.
\item[(ii)] Adaptive attitude tracking control system presented at Proposition \ref{prop:AT} with the following disturbances:
\begin{align*}
\Delta = 0.1\begin{bmatrix}\sin(2\pi t)& \cos(5\pi t) & R_{11}(t)\end{bmatrix}.
\end{align*}
\item[(iii)] Robust adaptive attitude tracking control system presented at Proposition \ref{prop:RAC} with the above disturbance model.
\end{itemize}

It has been shown that general-purpose numerical integrators fail to preserve the structure of the special orthogonal group $\SO$, and they may yields unreliable computational results for complex maneuvers of rigid bodies~\cite{IseMunAN00,HaiLub00}. In this paper, we use a geometric numerical integrators, referred to as a Lie group variational integrator, to preserve the underlying geometric structures of the attitude dynamics accurately~\cite{LeeLeoCMDA07}.

\begin{figure}
\centerline{
	\subfigure[Attitude error vector $e_R$]{
		\includegraphics[width=0.505\columnwidth]{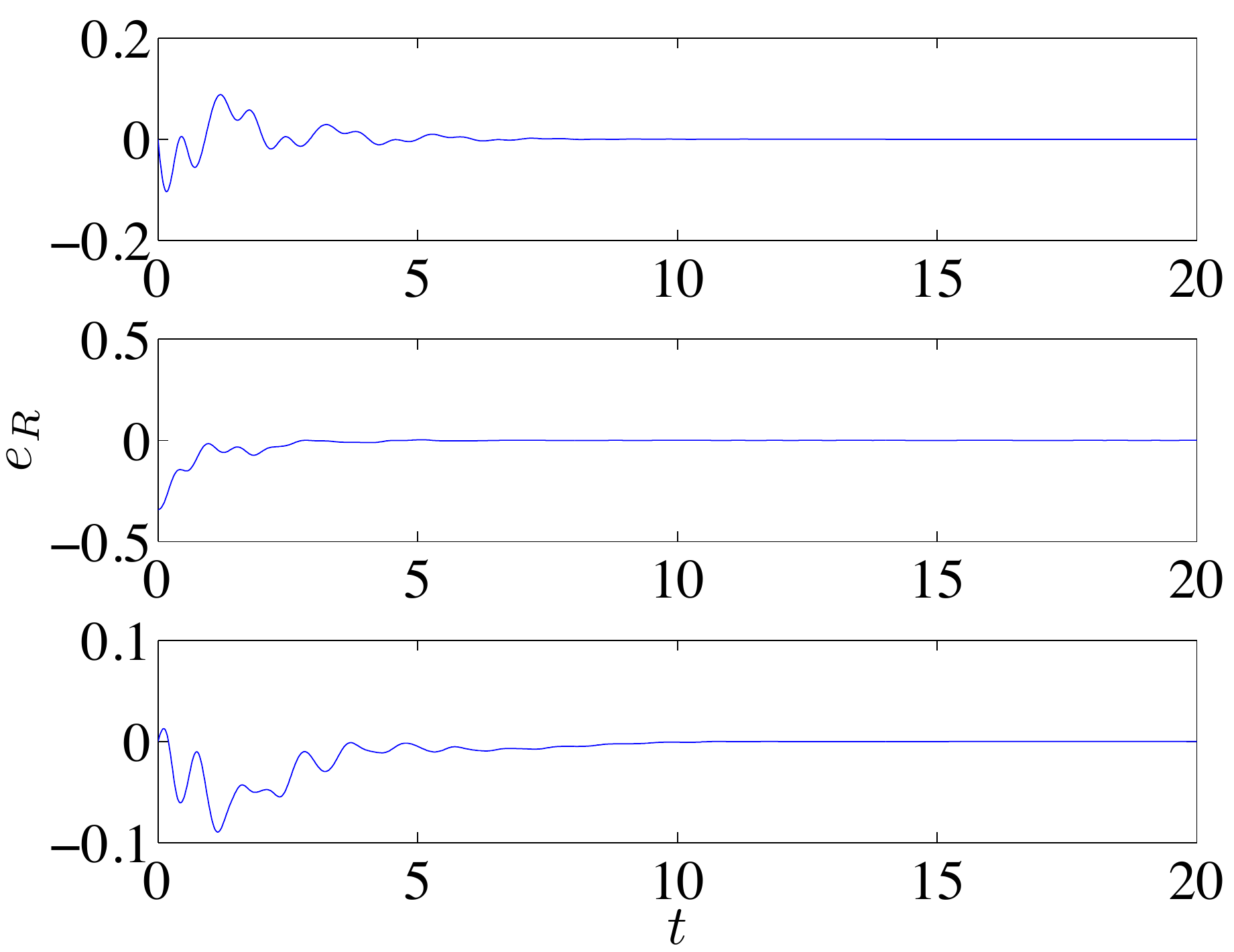}}
	\subfigure[Angular velocity ($\Omega$:blue, $\Omega_d$:red)]{
		\includegraphics[width=0.495\columnwidth]{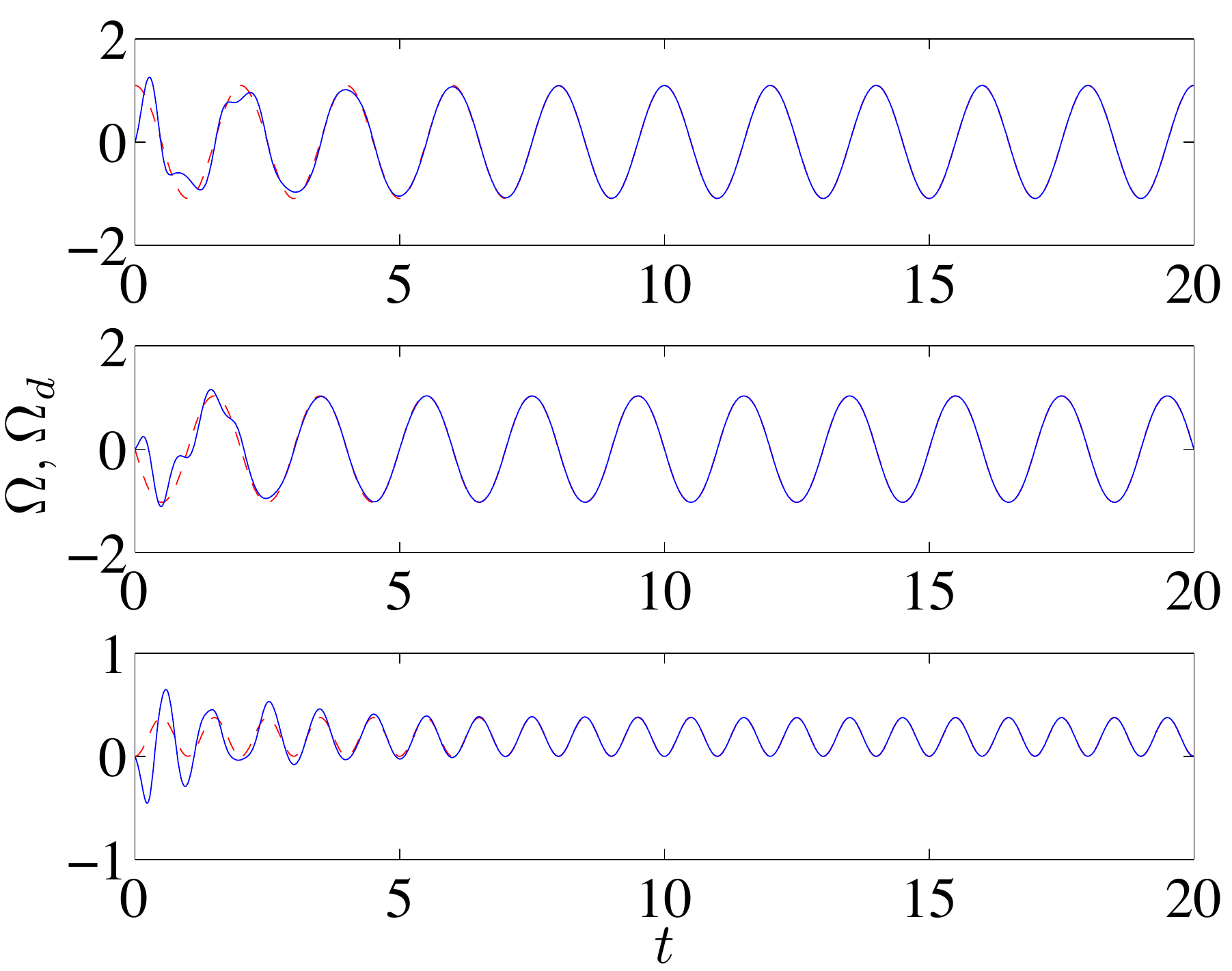}}
}
\centerline{
	\subfigure[Inertia estimate $\bar J$ ($\bar J_{11},\bar J_{12}$:solid, $\bar J_{22},\bar J_{13}$:dashed, $\bar J_{33},\bar J_{23}$:dotted)]{
		\includegraphics[width=0.5\columnwidth]{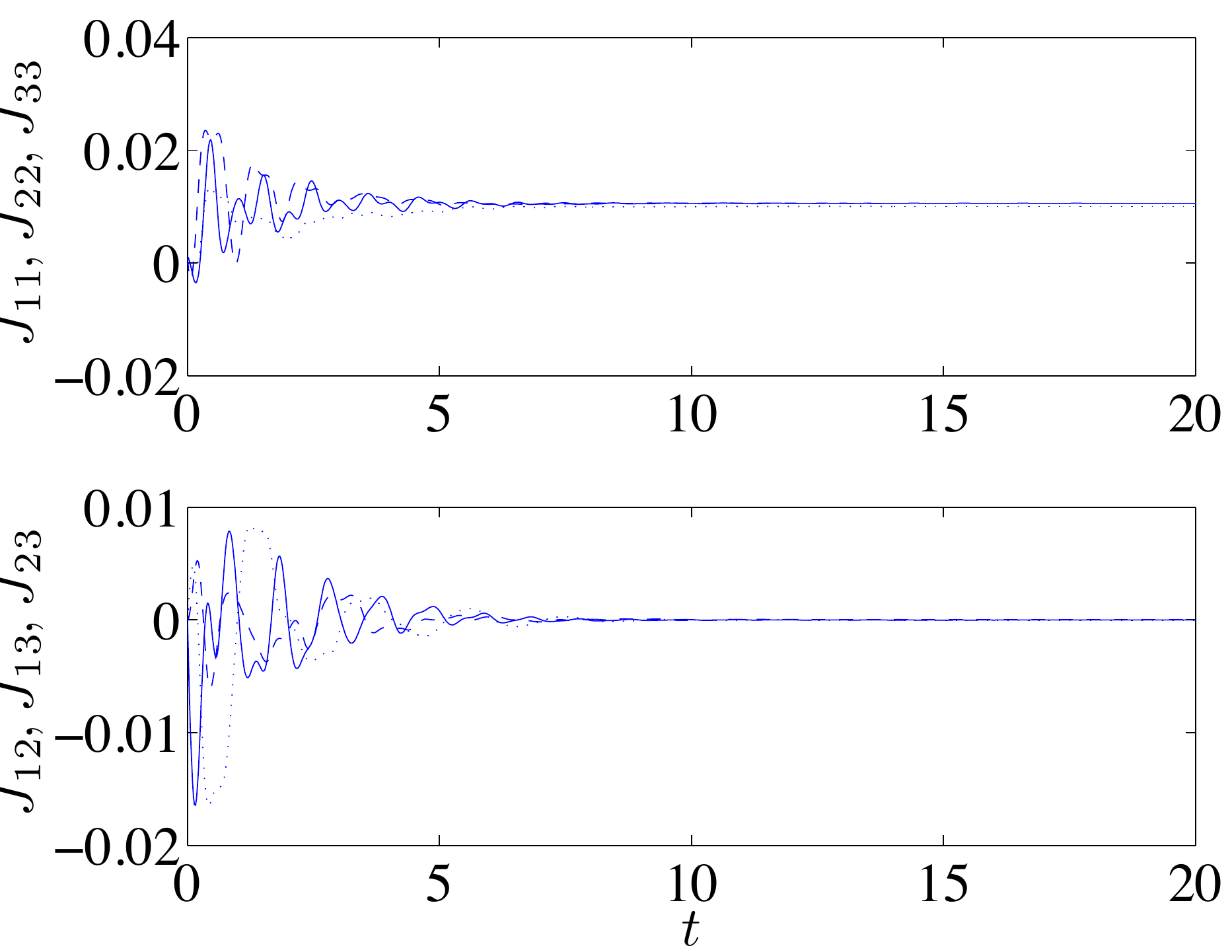}}
	\subfigure[Control input $u$]{
		\includegraphics[width=0.5\columnwidth]{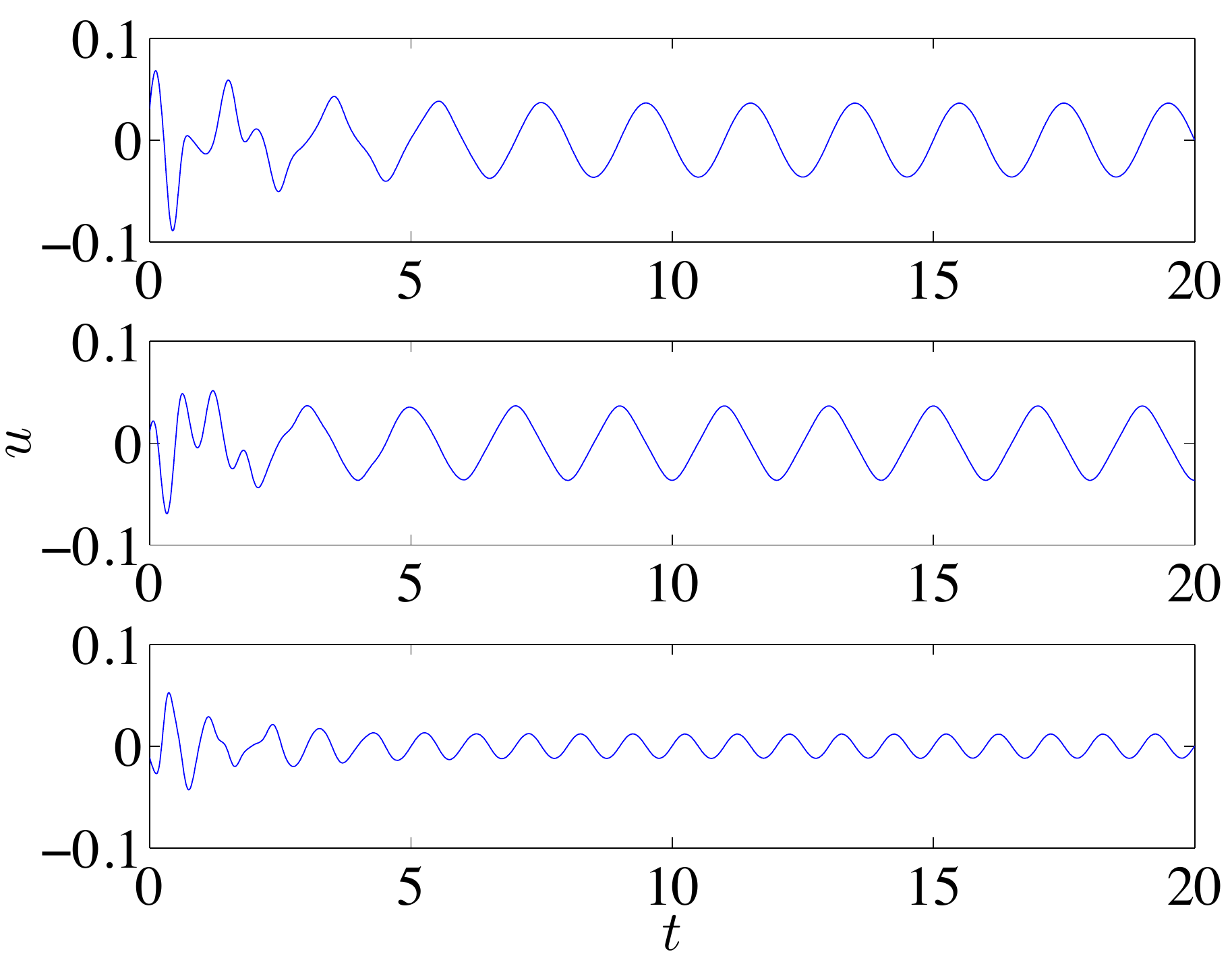}}
}
\caption{Adaptive attitude tracking without disturbances}\label{fig:1}
\end{figure}
\begin{figure}
\centerline{
	\subfigure[Attitude error vector $e_R$]{
		\includegraphics[width=0.50\columnwidth]{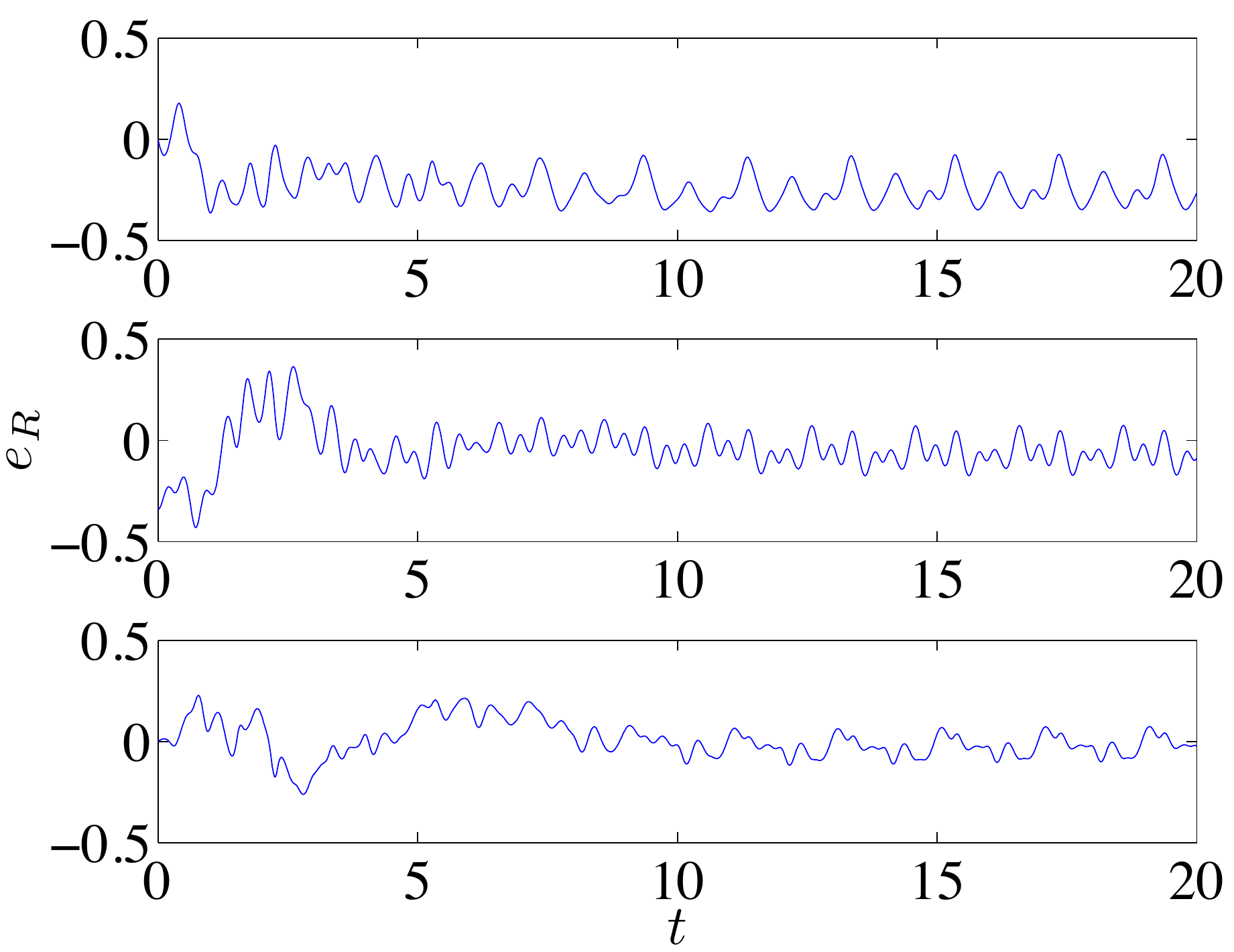}}
	\subfigure[Angular velocity ($\Omega$:blue, $\Omega_d$:red)]{
		\includegraphics[width=0.495\columnwidth]{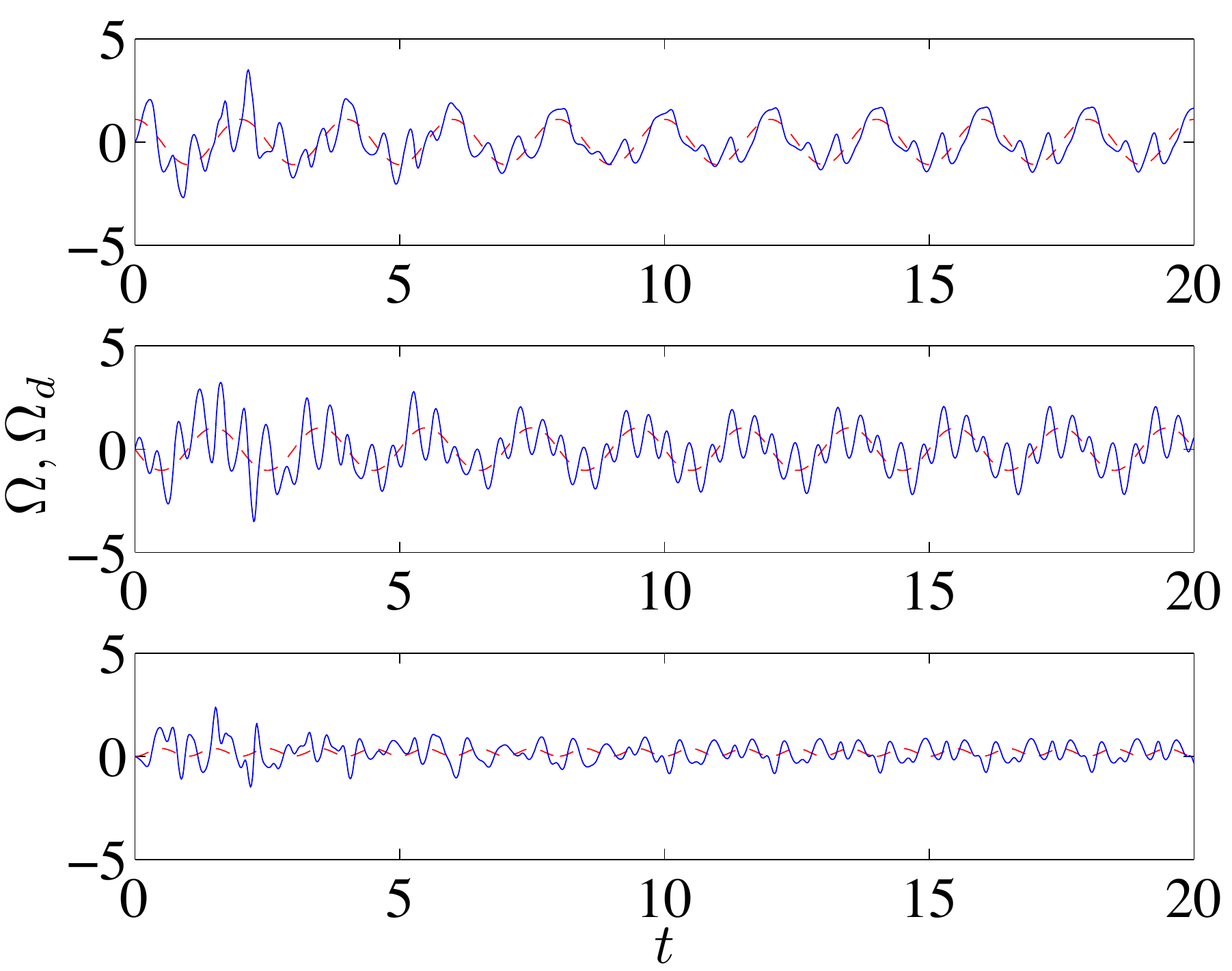}}
}
\centerline{
	\subfigure[Inertia estimate $\bar J$ ($\bar J_{11},\bar J_{12}$:solid, $\bar J_{22},\bar J_{13}$:dashed, $\bar J_{33},\bar J_{23}$:dotted)]{
		\includegraphics[width=0.5\columnwidth]{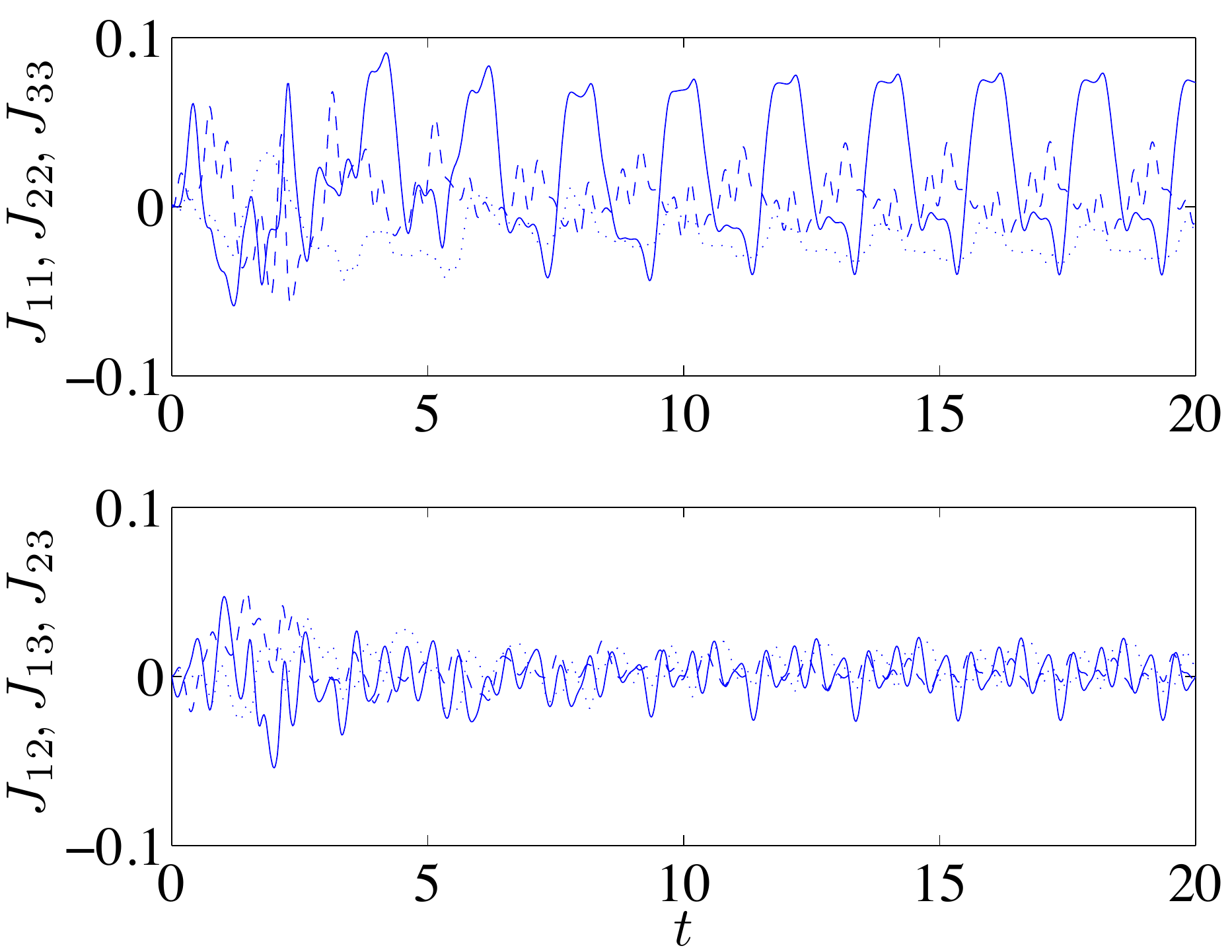}}
	\subfigure[Control input $u$]{
		\includegraphics[width=0.5\columnwidth]{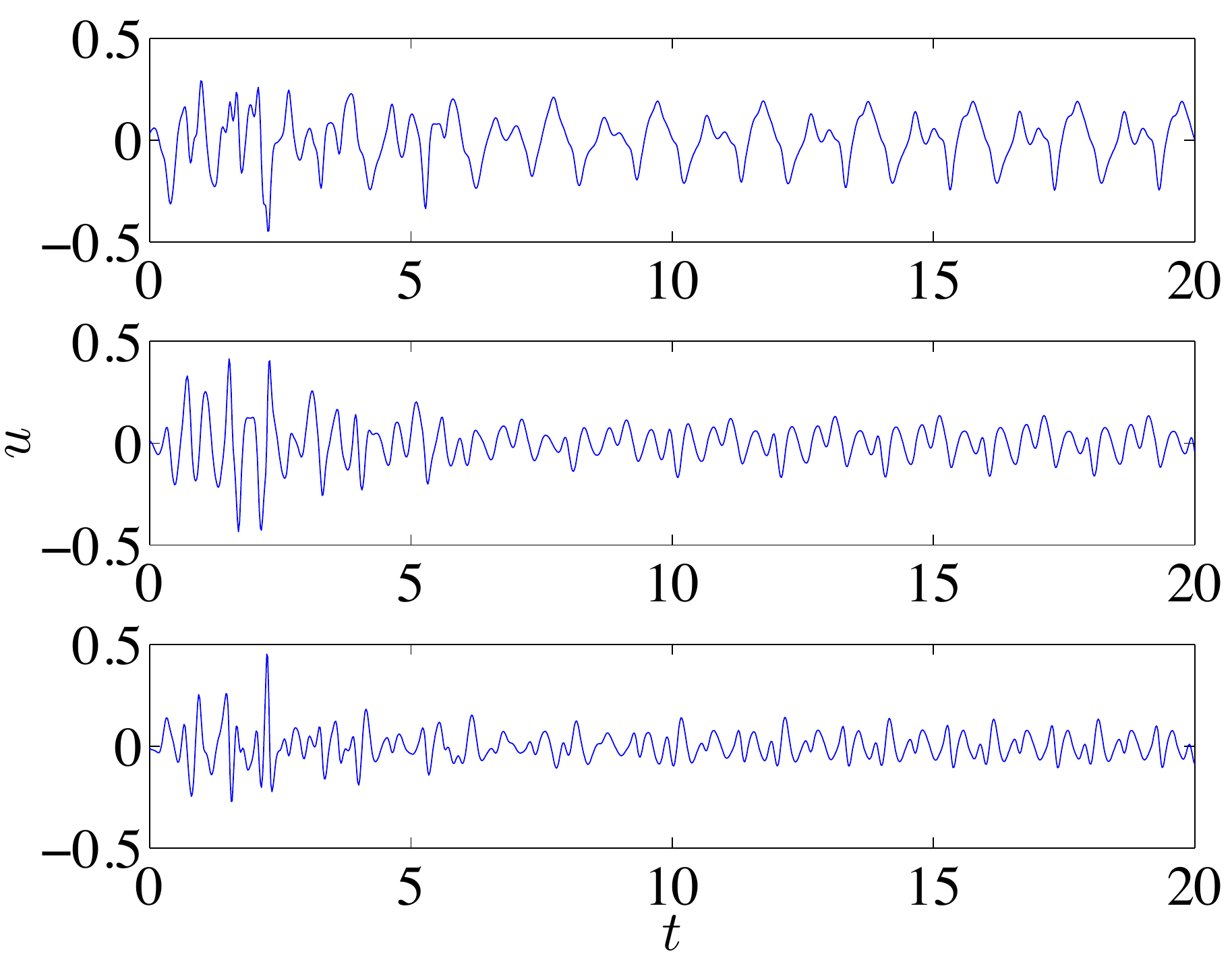}}
}
\caption{Adaptive attitude tracking with disturbances}\label{fig:2}
\end{figure}
\begin{figure}
\centerline{
	\subfigure[Attitude error vector $e_R$]{
		\includegraphics[width=0.52\columnwidth]{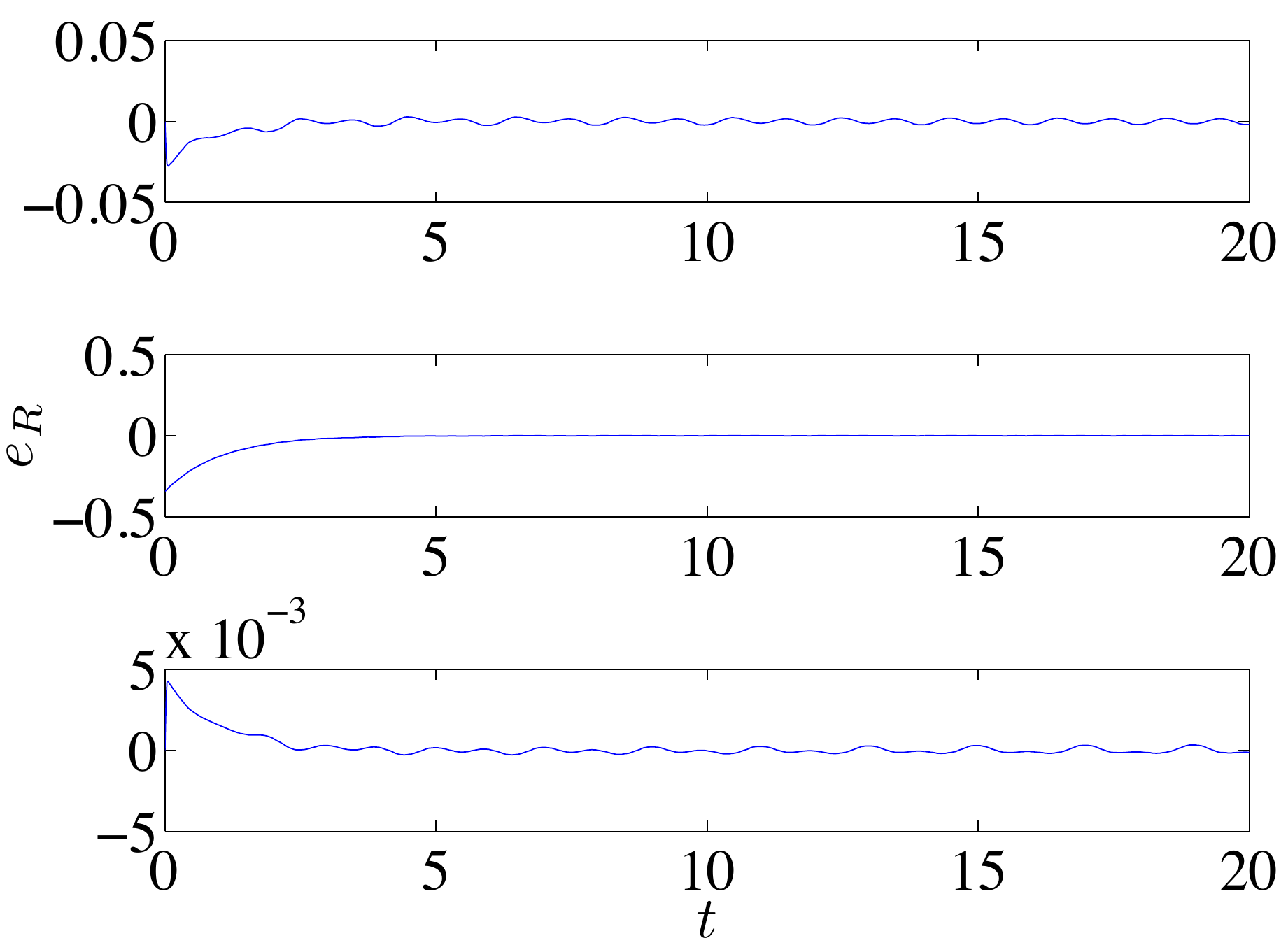}}
	\subfigure[Angular velocity ($\Omega$:blue, $\Omega_d$:red)]{
		\includegraphics[width=0.485\columnwidth]{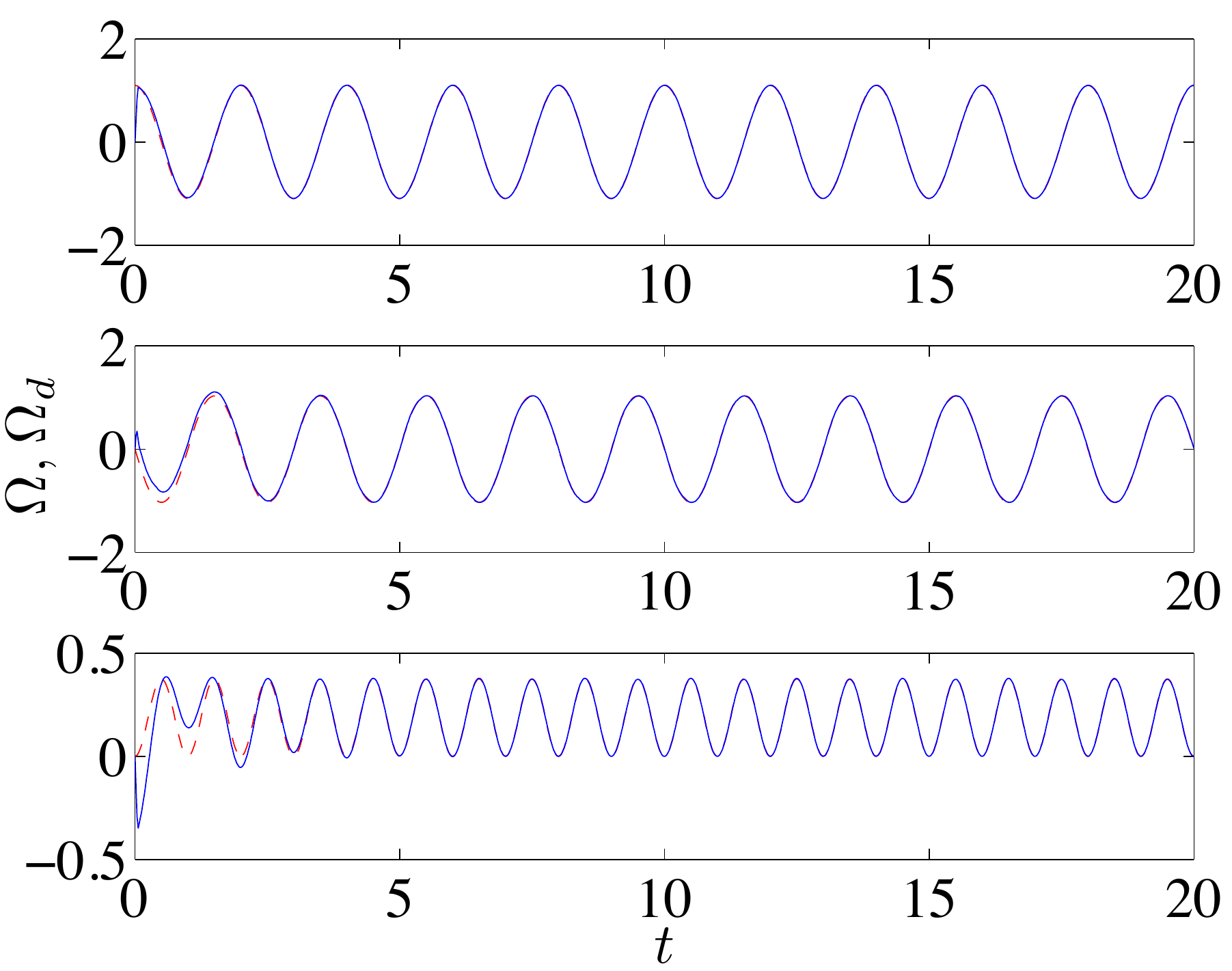}}
}
\centerline{
	\subfigure[Inertia estimate $\bar J$ ($\bar J_{11},\bar J_{12}$:solid, $\bar J_{22},\bar J_{13}$:dashed, $\bar J_{33},\bar J_{23}$:dotted)]{
		\includegraphics[width=0.51\columnwidth]{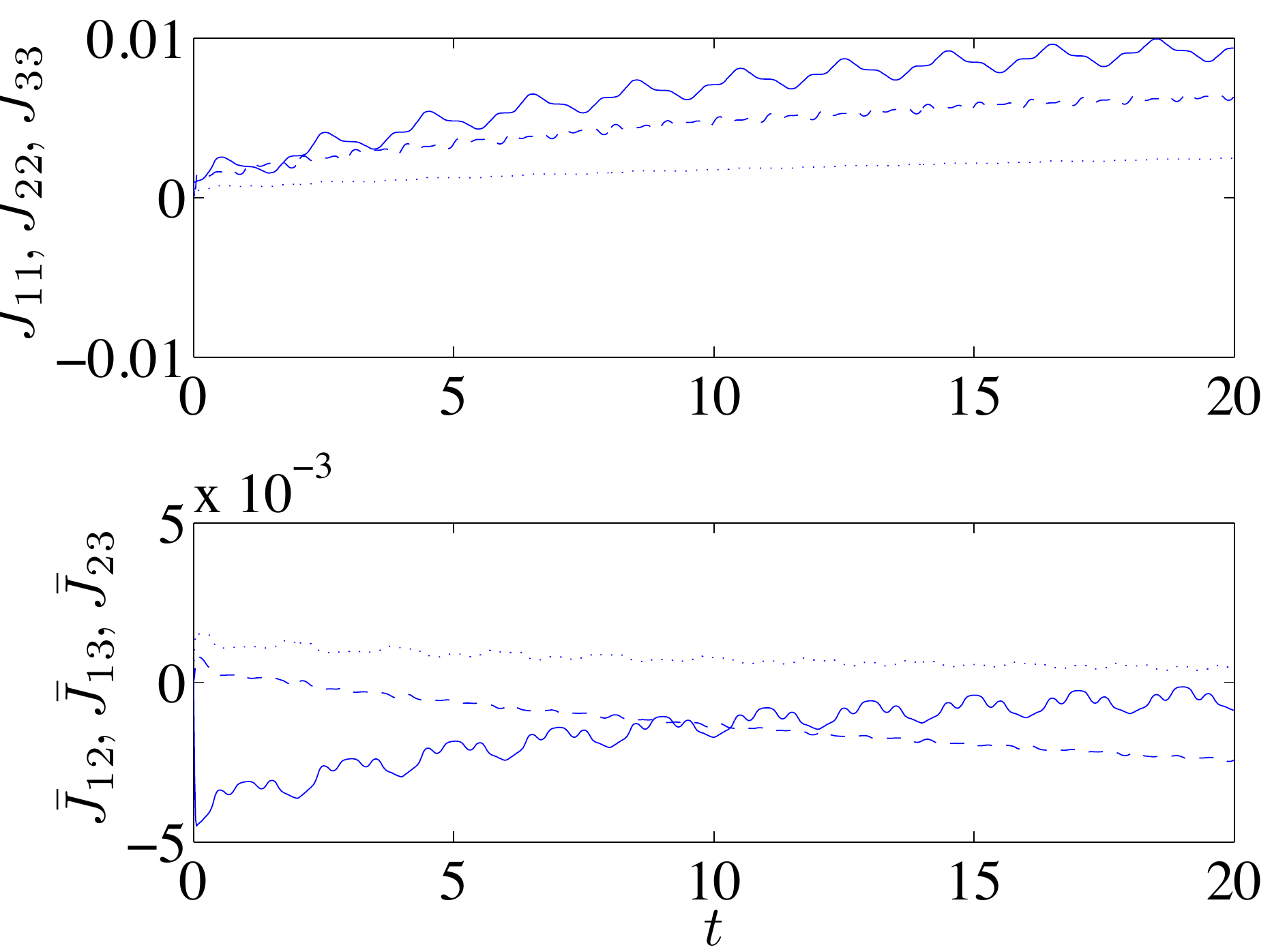}}
	\subfigure[Control input $u$]{
		\includegraphics[width=0.495\columnwidth]{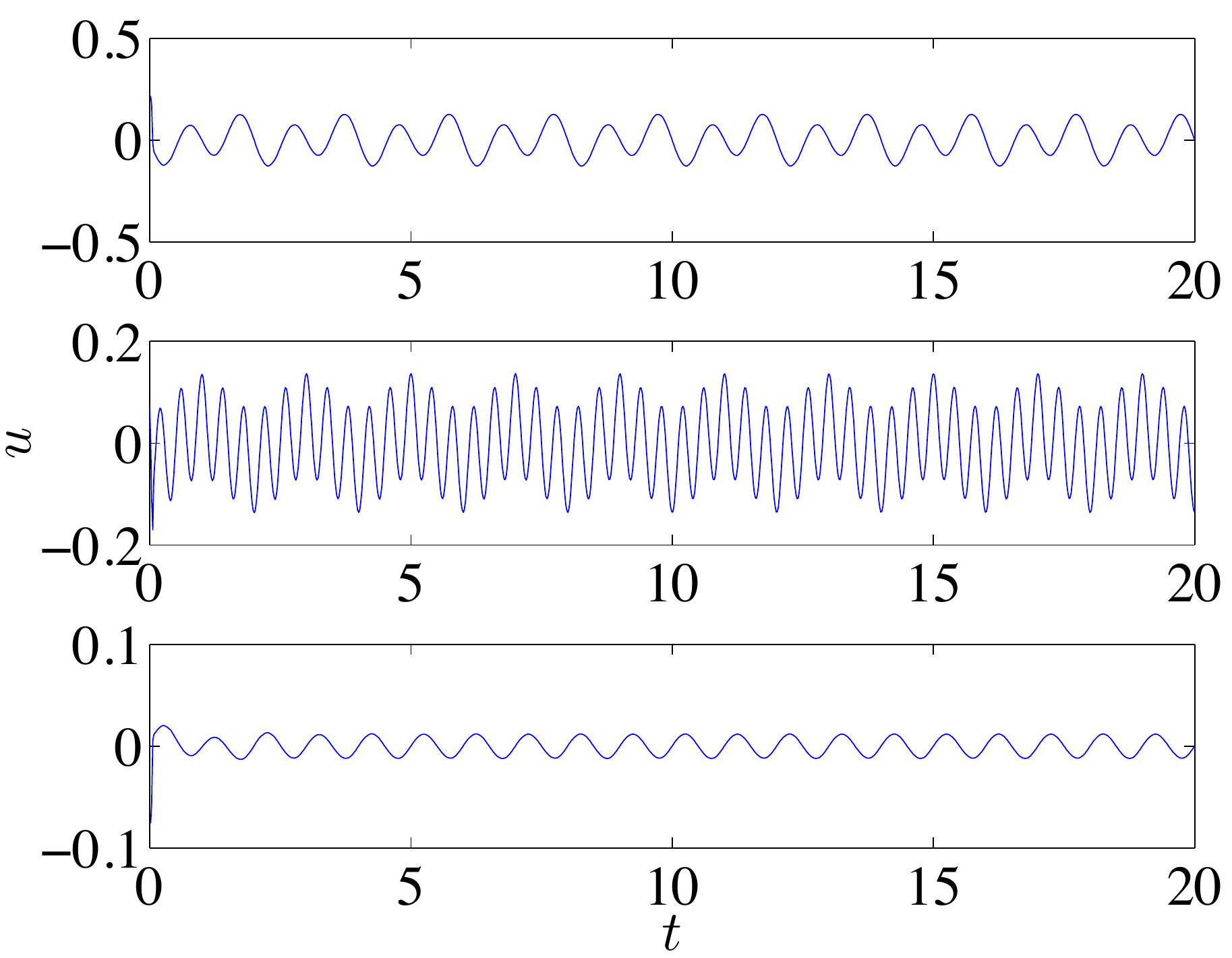}}
}
\caption{Robust adaptive attitude tracking with disturbances}\label{fig:3}
\end{figure}

Simulation results are illustrated at Figures \ref{fig:1}-\ref{fig:3}. When there is no disturbance, the adaptive attitude tracking control system presented at Proposition \ref{prop:AT} follows the given  attitude command accurately while making the estimate of the inertia matrix converge to a fixed matrix at \reffig{1}. But, these convergence properties are degraded in the presence of disturbances. At \reffig{2}, the tracking errors are not converged to zero asymptotically, and the estimate of the inertia matrix and control inputs fluctuate. These are significantly improved by the robust adaptive tracking controller discussed at Proposition \ref{prop:RAC}. At \reffig{3}, the tracking errors for the attitude and the angular velocity are close to zero, and the estimate of the inertia matrix is bounded. These show that the proposed robust adaptive approach is critical in following an attitude command in the presence of disturbances.

\section{Experiment on a Quadrotor UAV}

A quadrotor unmanned aerial vehicle (UAV) is composed of two pairs of counter-rotating rotors and propellers. Due to its simple mechanical structure, it has been envisaged for various applications such as surveillance or mobile sensor networks as well as for educational purposes.

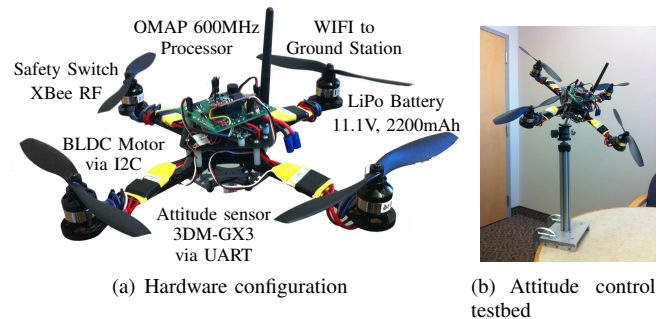
\begin{figure}[b]
\centerline{
	\subfigure[Hardware configuration]{
\setlength{\unitlength}{0.1\columnwidth}\scriptsize
\begin{picture}(7,4)(0,0)
\put(0,0){\includegraphics[width=0.7\columnwidth]{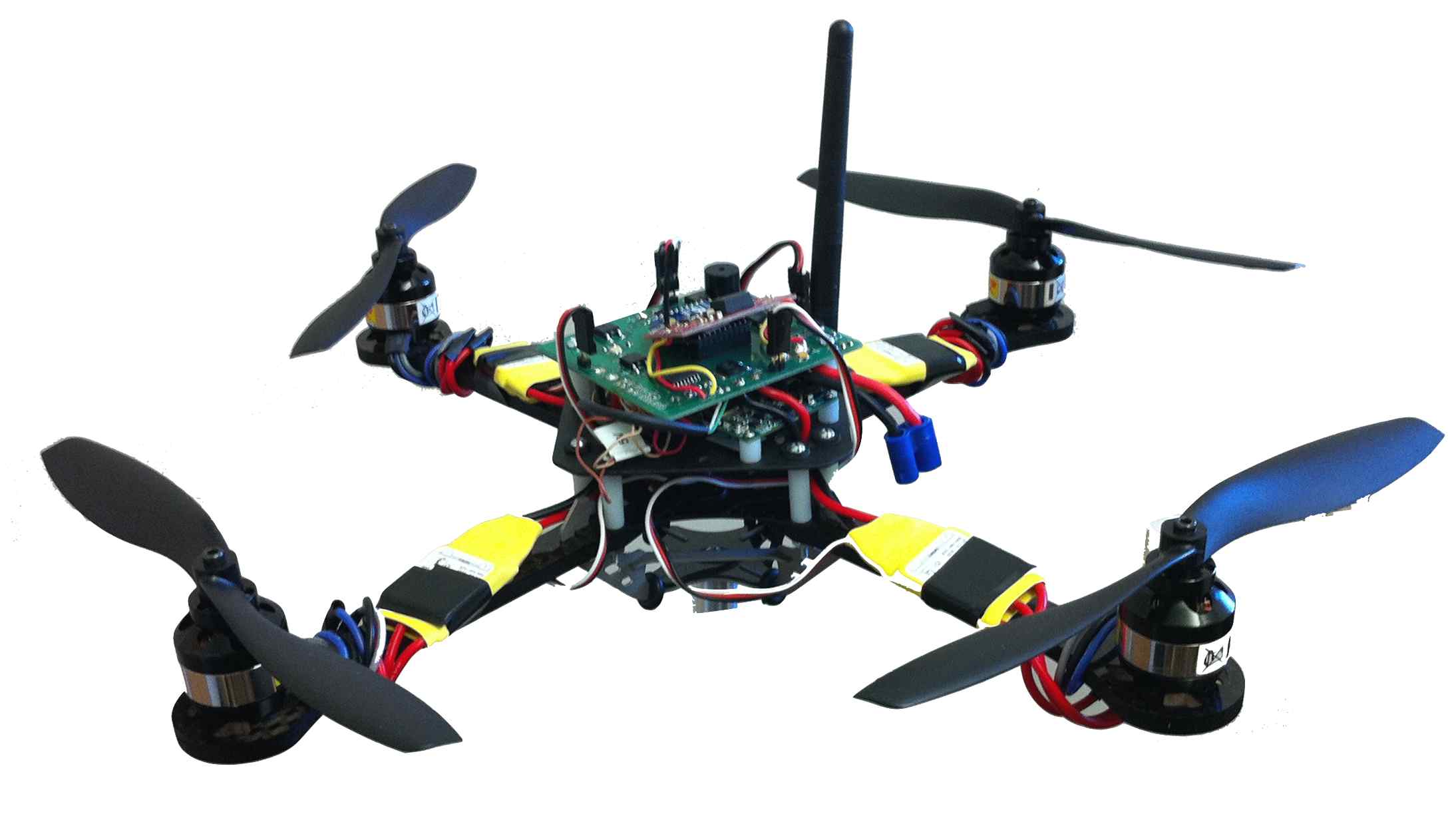}}
\put(1.95,3.2){\shortstack[c]{OMAP 600MHz\\Processor}}
\put(2.3,0){\shortstack[c]{Attitude sensor\\3DM-GX3\\ via UART}}
\put(0.85,1.4){\shortstack[c]{BLDC Motor\\ via I2C}}
\put(0.1,2.5){\shortstack[c]{Safety Switch\\XBee RF}}
\put(4.3,3.2){\shortstack[c]{WIFI to\\Ground Station}}
\put(5,2.0){\shortstack[c]{LiPo Battery\\11.1V, 2200mAh}}
\end{picture}}
	\subfigure[Attitude control testbed]{
	\includegraphics[width=0.27\columnwidth]{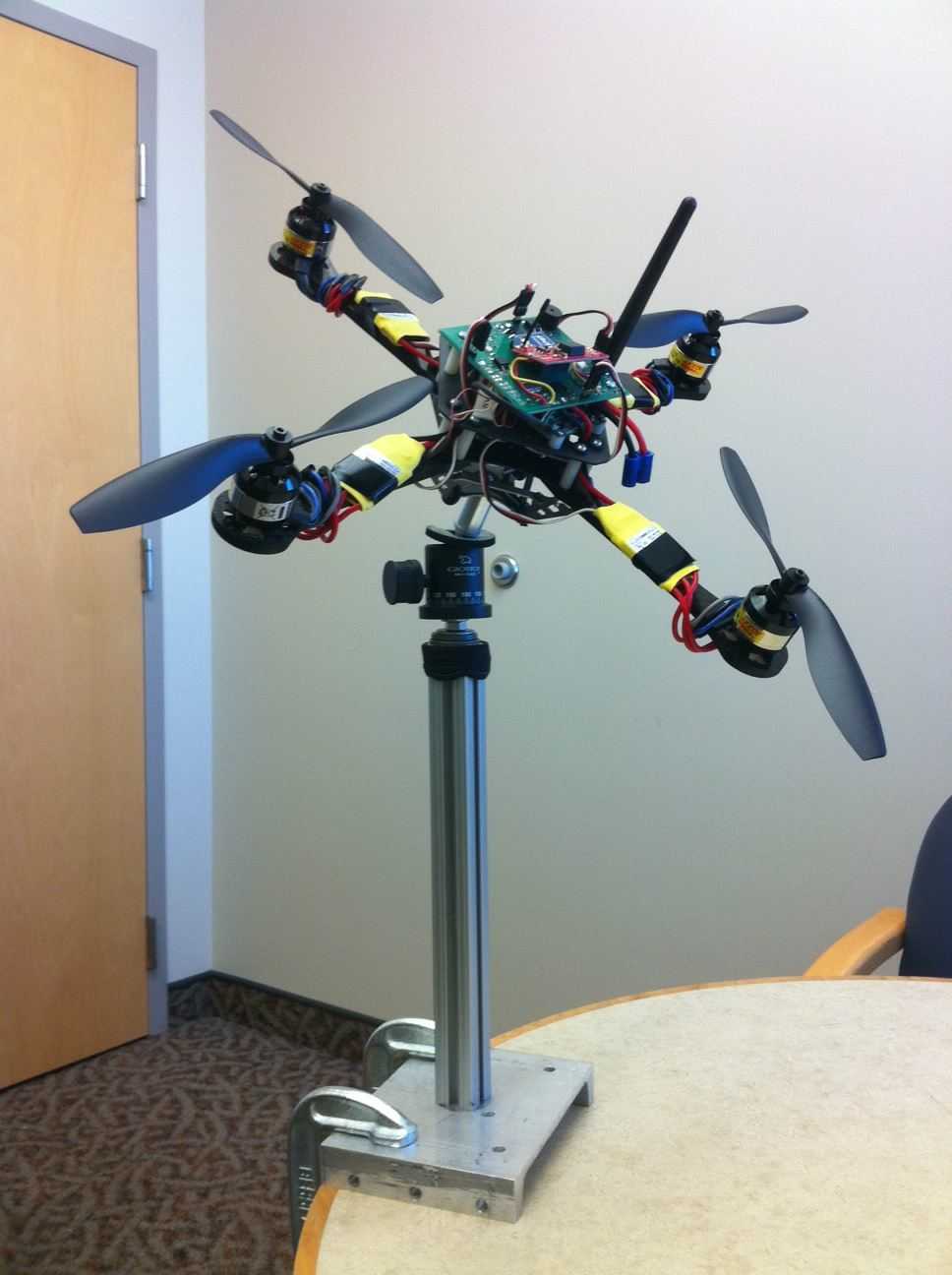}}
}
\caption{Attitude control experiment for a quadrotor UAV}
\end{figure}

We have developed a hardware system for a quadrotor UAV. It is composed of the following parts:
\begin{itemize}
\item Gumstix Overo computer-in-module (OMAP 600MHz processor), running a non-realtime Linux operating system. It is connected to a ground station via WIFI.
\item Microstrain 3DM-GX3 attitude sensor, connected to Gumstix via UART.
\item Phifun motor speed controller, connected to Gumstix via I2C.
\item Roxxy 2827-35 Brushless DC motors.
\item MaxStream XBee RF module, which is used for an extra safety switch.
\end{itemize}
To test the attitude dynamics only, it is attached to a spherical joint. As the center of rotation is below the center of gravity, there exists a destabilizing gravitational moment, and the resulting attitude dynamics is similar to an inverted rigid body pendulum. 

We apply the robust adaptive attitude control system at Proposition \ref{prop:RAC} to this quadrotor UAV. The control input at \refeqn{u2} is augmented with an additional term to eliminate the gravitational moment. The disturbances are mainly due to the error in canceling the gravitational moment, the friction in the spherical joint, as well as sensor noises and thrust measurement errors. 

The attitude tracking command and control input parameters are identical to the numerical examples discussed in the previous section, except the following variables:
\begin{gather*}
k_J=0.01,\quad \sigma=0.01,\quad \epsilon=0.35.
\end{gather*}

The corresponding experimental results are illustrated at \reffig{4}. Overall, it exhibits a good attitude command tracking performance, while the second component of the attitude error vector $e_R$, and the third component of the angular velocity tracking error are relatively large. The estimates of the inertia matrix are bounded (a video clip showing the controlled attitude maneuver is available at \href{http://my.fit.edu/\~taeyoung/Animation/QuadRAC.mov}{http://my.fit.edu/$\sim$taeyoung/Animation/QuadRAC.mov}).

\begin{figure}
\centerline{
	\subfigure[Attitude error vector $e_R$]{
		\includegraphics[width=0.505\columnwidth]{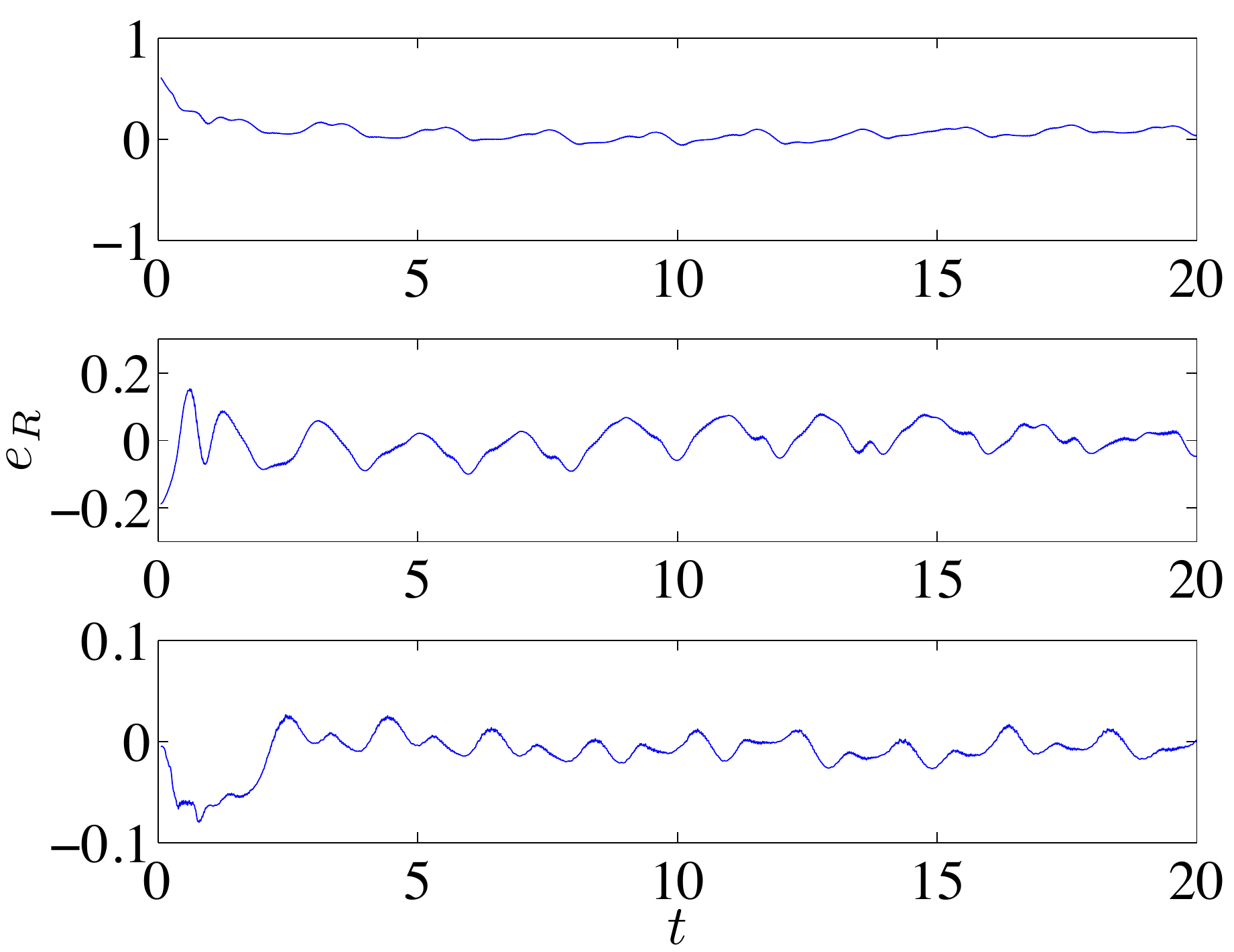}}
	\subfigure[Angular velocity ($\Omega$:blue, $\Omega_d$:red)]{
		\includegraphics[width=0.495\columnwidth]{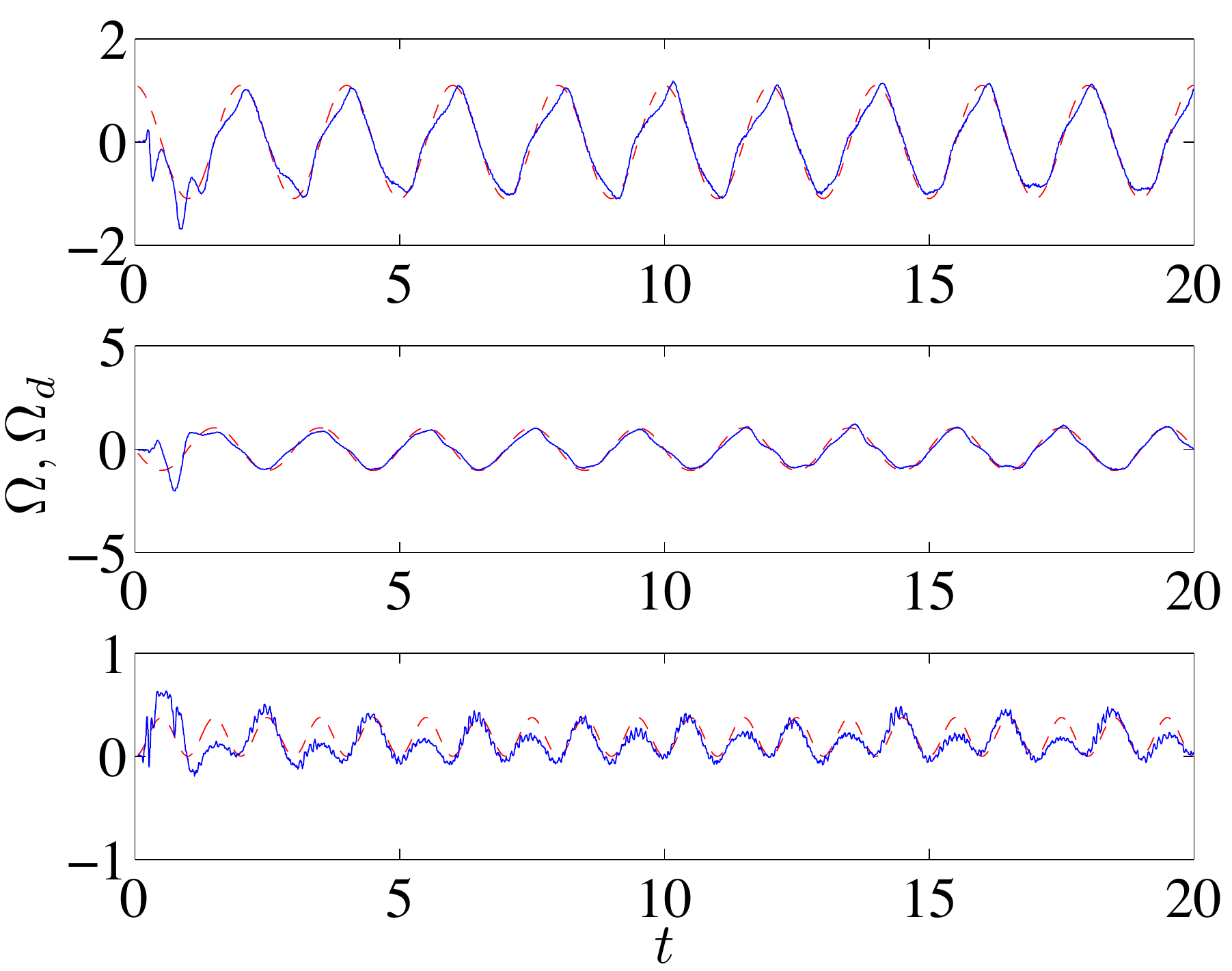}}
}
\centerline{
	\subfigure[Inertia estimate $\bar J$ ($\bar J_{11},\bar J_{12}$:solid, $\bar J_{22},\bar J_{13}$:dashed, $\bar J_{33},\bar J_{23}$:dotted ($\mathrm{kgm^2}$))]{
		\includegraphics[width=0.51\columnwidth]{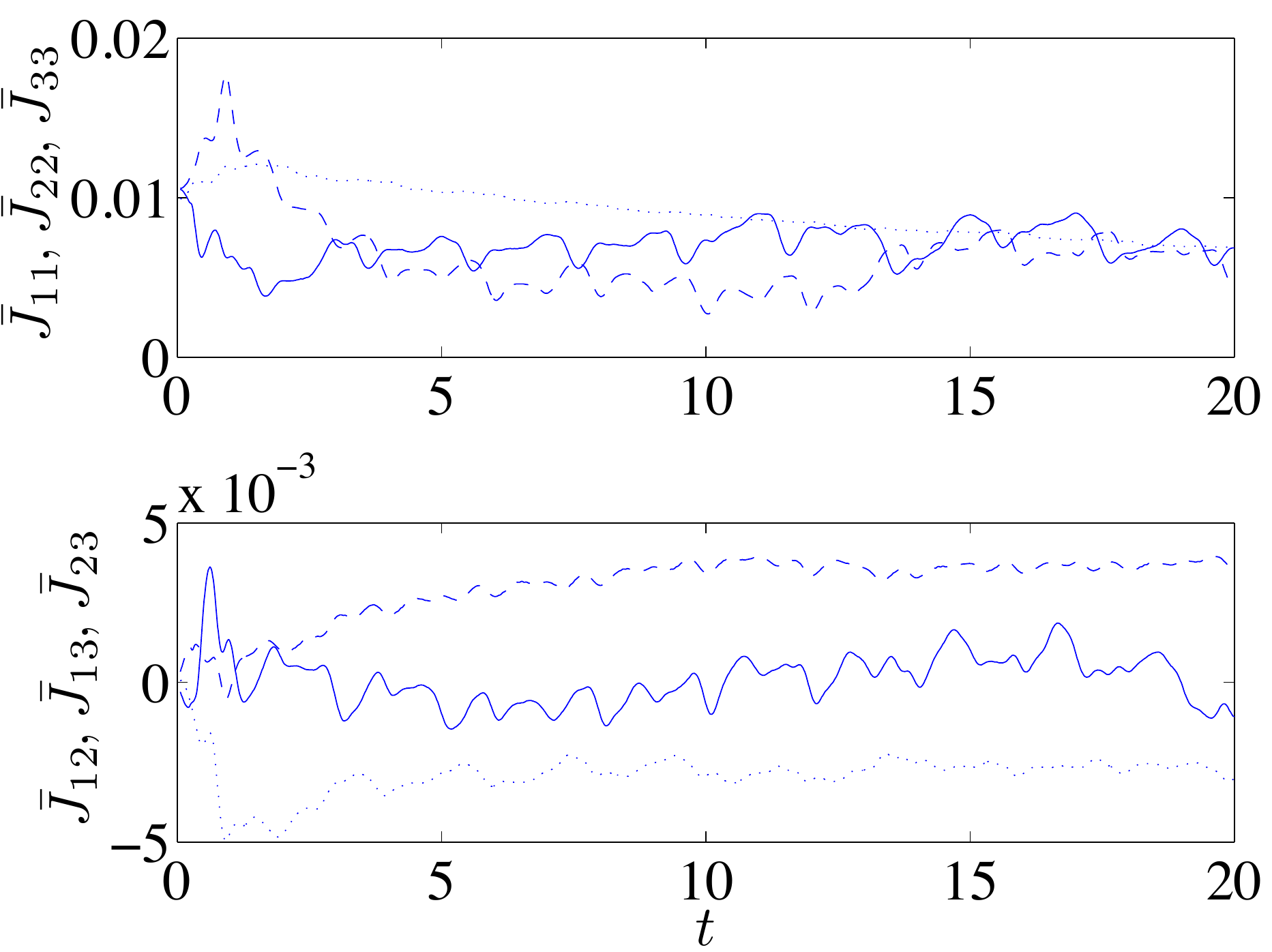}}
	\subfigure[Control input $u$]{
		\includegraphics[width=0.49\columnwidth]{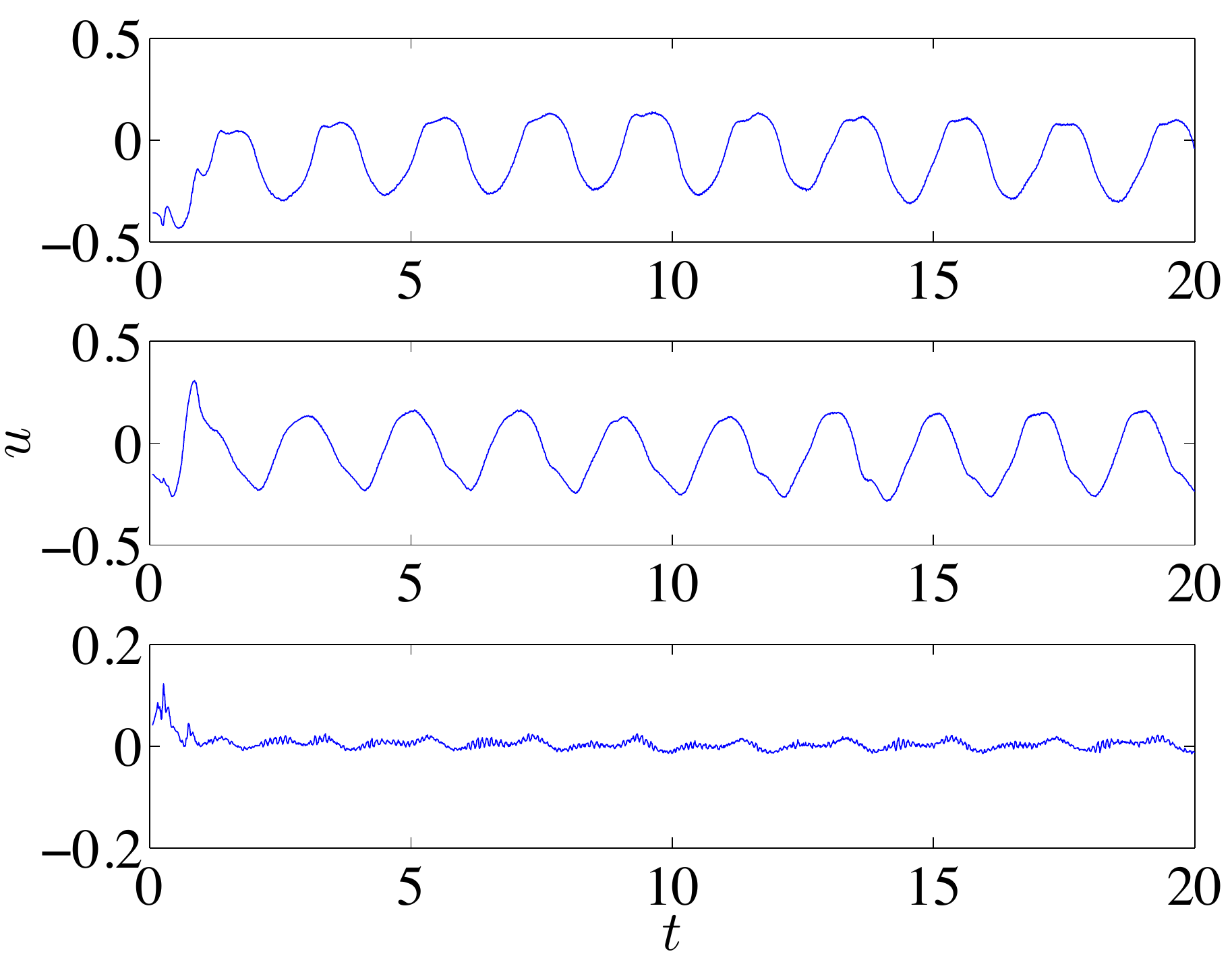}}
}
\caption{Robust adaptive attitude tracking experiment}\label{fig:4}
\end{figure}

\section{Conclusion}

We have developed adaptive tracking control systems on $\SO$. The proposed control system is constructed directly on $\SO$ to avoid singularities and ambiguities that are inherent to other attitude representations. A adaptive control system is developed to asymptotically follow a given attitude tracking command without the knowledge of an inertia matrix, in the absence of disturbances. A robust adaptive control system is proposed to eliminate the effects of disturbances. These properties are illustrated by numerical examples and a hardware experiment of the attitude dynamics of a quadrotor UAV.

\section*{Acknowledgment}
The author wishes to thank Thilina Fernando and Jiten Chandiramani for their help in the development of the presented quadrotor UAV hardware system.

\bibliography{CDC11.2}
\bibliographystyle{IEEEtran}

\end{document}